\documentclass[10pt,letterpaper,reqno]{amsart}

\usepackage{amssymb}

\usepackage{xcolor}

\usepackage{thmtools}
\usepackage{thm-restate}
\usepackage{cleveref}
\usepackage{comment}
\usepackage{enumitem}
\usepackage{comment}
\usepackage[backend=biber,style=alphabetic,doi=false,isbn=false,url=false,eprint=false]{biblatex}
\newtheorem{theorem}{Theorem}[section]
\newtheorem{lemma}[theorem]{Lemma}
\newtheorem{corollary}[theorem]{Corollary}
\newtheorem{proposition}[theorem]{Proposition}
\newtheorem{question}[theorem]{Question}

\theoremstyle{definition}
\newtheorem{definition}[theorem]{Definition}

\theoremstyle{remark}
\newtheorem{remark}[theorem]{Remark}

\numberwithin{equation}{section}

\newcommand{\transv}{\mathrel{\text{\tpitchfork}}}
\makeatletter
\newcommand{\tpitchfork}{%
  \vbox{
    \baselineskip\z@skip
    \lineskip-.52ex
    \lineskiplimit\maxdimen
    \m@th
    \ialign{##\crcr\hidewidth\smash{$-$}\hidewidth\crcr$\pitchfork$\crcr}
  }%
}
\makeatother

\DeclareMathOperator{\img}{Im}
\newcommand{\tr}{\mathrm{tr}}

\newcommand{\tcodim}{transversality codimension}

\begin{document}

\title[Arnold's transversality conjecture for the Laplacian]{On Arnold's Transversality Conjecture for the Laplace--Beltrami Operator}

\author{Josef Greilhuber}
\address{380 Jane Stanford Way, Stanford, CA 94305, USA}
\curraddr{}
\email{jgreil@stanford.edu}
\thanks{}

\author{Willi Kepplinger}
\address{Oskar Morgenstern Platz 1, 1090 Wien, Austria}
\curraddr{}
\email{willi.kepplinger@univie.ac.at}
\thanks{}


\date{\today}

\dedicatory{}

\begin{abstract}
This paper is concerned with the structure of the set of Riemannian metrics on a connected manifold such that the corresponding Laplace-Beltrami operator has an eigenvalue of a given multiplicity. 
The starting point of our investigation is the ``Strong Arnold Hypothesis'' introduced by Colin de Verdi\`{e}re, which posits that Laplace eigenvalues of higher multiplicity split up under perturbation exactly as eigenvalues of symmetric matrices do. There exists a simple geometric characterization, due to Colin de Verdi\`ere and Besson, of metrics which satisfy the Strong Arnold Hypothesis.

Using Besson's characterization, we prove the Strong Arnold Hypothesis is satisfied for all metrics except for a set of infinite codimension, and use this to obtain the precise codimension of the set of metrics admitting an eigenvalue of any given multiplicity. Furthermore, we show that the Strong Arnold Hypothesis is satisfied for all metrics admitting eigenvalues of multiplicity at most six, and discuss several examples of metrics violating it.
\end{abstract}

\maketitle

\section{Introduction}

A famous result of Wigner and von Neumann \cite{VonNeumannWigner} called the ``non-crossing rule'' states that energy levels of quantum mechanical systems do not cross along generic $1$-parameter families of Hamiltonians. Inspired by this, Arnold \cite{Arnold1972} proposed a transversality condition for sufficiently general families of operators, later called the \emph{strong Arnold hypothesis} (SAH) by Colin de Verdi\`{e}re \cite{CdV1988}: operators in the family with an eigenvalue of multiplicity $m$ should locally form a manifold of codimension $\frac{m(m+1)}{2}-1$, given by a natural defining function. Arnold also noted that generic $k$-parameter families of such operators avoid $m$-fold eigenvalues to exactly the extent one would expect by the codimension count. We will say that such a family of self-adjoint operators satisfies a \emph{maximal non-crossing rule}.

The non-crossing rule for $1$-parameter families of Laplace-Beltrami operators was proven by Uhlenbeck \cite{Uhlenbeck1976} using the Sard--Smale theorem, as opposed to a local structure theorem for the set of operators with multiple eigenvalues that Arnold had in mind. A perturbation-theoretic approach closer to that proposed by Arnold was taken up in 1980 by Bleecker and Wilson \cite{BleeckerWilson1980} and 1983 by Bando and Urakawa \cite{BandoUrakawa1983}. In both cases, it was used to prove generic simplicity of the Laplace-Beltrami spectrum in a fixed conformal class. Arnold had conjectured that the family of Laplace operators over the set of Riemannian metrics on a given manifold satisfies the SAH, which would yield a much stronger codimension statement for eigenvalues of higher multiplicity, but counterexamples of metrics not satisfying the SAH were exhibited by Colin de Verdi\`{e}re in \cite{CdV1988}. In the same article, the first explicit criterion for the failure of the SAH was obtained. This criterion, which was subsequently greatly simplified by Besson \cite{Besson1989}, is the starting point of the present paper.

To the best of our knowledge, there has not been further progress on understanding the local structure of the set of Riemannian metrics for which the Laplace operator has eigenvalues of multiplicity greater than two. In particular, no non-crossing rule in multiplicity greater than two has been established for the Laplacian (nor for any other natural, infinite dimensional family of operators). In this paper we prove that the set of metrics for which the SAH does not hold has infinite codimension. This means that Arnold's transversality conjecture ``morally'' holds for this family of self-adjoint operators, because it allows for a maximal non-crossing rule: Generic $k$-parameter families of metrics avoid Laplace operators with multiple eigenvalues to exactly the extent that Arnold expected.\par

Finally, let us mention several investigations along similar lines. Lupo and Micheletti \cite{LupoMicheletti1993} independently discovered Arnold's approach to eigenvalue multiplicities in families of self-adjoint operators. Subsequently, Micheletti and Pistoia \cite{Micheletti1998} showed that the SAH is satisfied for Laplace operators on $2$-dimensional manifolds for eigenvalues of multiplicity $2$, which implies Uhlenbeck's non-crossing rule via the perturbation-theoretic approach.

Around the same time, Teytel formalized a useful \emph{abstract} criterion for establishing $1$-parameter non-crossing rules \cite[Theorem A]{Teytel1999}, conditional on the SAH. In a similar vein, Porter and Schwarz \cite{PorterSchwarz2017} established very general theorems on avoidance of intersection with countable unions of submanifolds in a Banach space.

More recent work on the problems considered in this paper includes an alternative proof of generic simplicity in a conformal class from the variational principle for the Rayleigh quotient by Guillemin, Legendre and Sena-Dias \cite{GuilleminLegendre2014}, who also note the failure of transversality observed by Colin de Verdi\`ere, and Gomes and Marrocos' proof \cite{GomesMarrocos2019} of generic simplicity of the Laplace--Neumann spectrum.

A significant body of work is devoted to the study of eigenvalue simplicity in the presence of a group action, a setting in which the SAH is almost never satisfied, starting with Zelditch' \cite{Zelditch1990} study of generic covers of Riemannian manifolds. Here, the analogue of generic simplicity of eigenspaces is \emph{generic irreducibility} of the eigenspaces viewed as representations of the group in question. Recent work includes papers of Schueth and others (\cite{Schueth2017}, \cite{PetreccaRoser}, \cite{MarrocosDeOliveira2024}) devoted to invariant metrics on homogeneous or symmetric spaces, Jung and Zelditch's analysis of metrics with a circle action (\cite{JungZelditch2020}) and an investigation of torus actions by Cianci, Judge, Lin, and Sutton \cite{CianciJudgeLinSutton2024}.

Lastly, the notion of stable eigenvalues has been used to great effect in order to prescribe any finite part of the spectrum of several geometric operators. This was first done by Colin de Verdi\`{e}re \cite{CdV1986,CdV1987} for the Laplace operator (a result that was significantly refined by Lohkamp in \cite{Lohkamp1996}). Colin de Verdi\'{e}re's approach, which involves constructing a simplified model with the desired spectral properties and using transversality theory to show these persist in the setting at hand, was later  successfully applied to many more differential operators, such as to the Dirac operator by Dahl \cite{Dahl2005}, the Hodge--Laplace operator by Guerini \cite{Guerini2004} and Jammes \cite{Jammes2008}, the Steklov spectrum by Jammes \cite{Jammes2014} and the Laplace--Beltrami operator with Dirichlet boundary conditions by He \cite{He2024} and He--Wang \cite{HeWang2024}. In related work, Ann\'{e} derived surgery formulas for the Laplace spectrum \cite{Anne1994}.

\subsection*{Acknowledgements}
The first author is grateful to his advisor, Professor Eugenia Malinnikova, for her tireless support and wonderful mentoring. Special thanks go also to Professor Rafe Mazzeo for insightful comments putting \Cref{lemma: smooth kernel} into perspective, and to Romain Speciel and Ben Foster for being great friends and sources of inspiration in spectral theory and elliptic PDEs, respectively. 

The second author wishes to thank his advisor Vera V\'ertesi and Michael Eichmair for their constant support and helpful mentoring as well as the Vienna School of Mathematics for providing a stable and pleasant research environment. Much of the relevant research was conducted during a stay at Stanford University, and so he wishes to express his gratitude to the institution and, in particular, to Eugenia Malinnikova for the invitation. His research stay was made possible by NSF grant DMS-2247185.

The authors are deeply grateful to Christopher Judge for illuminating discussions around the background of this problem.

This research was funded in part by the Austrian Science Fund (FWF) [10.55776/ P34318] and [10.55776/Y963]. For open access purposes, the author has applied a CC BY public copyright license to any author-accepted manuscript version arising from this submission.

\subsection{Main results.}

Let $M$ be a closed, smooth manifold, which we will always assume to be of dimension at least $2$. Denote the Fr\'{e}chet manifold of smooth Riemannian metrics on $M$ by $\mathcal G(M)$. Fixing $g \in \mathcal G(M)$, we also consider the space $\mathcal G(M,g)$ of smooth conformal multiples of $g$. To each $g \in \mathcal G(M)$, we associate the corresonding Laplace--Beltrami operator $\Delta_g$. We will denote the space of real symmetric $m\times m$ matrices with $\mathbb R^{\frac{m(m+1)}{2}}$. 

In \cite{Besson1989}, Besson established the following theorem, building on work by Bleecker--Wilson and Colin de Verdi\`ere.

\begin{theorem}[\cite{Besson1989}]
\label{theorem:transverse structure theorem}
Let $g$ be a Riemannian metric on a closed smooth manifold $M$ of dimension at least $2$, and let $\lambda \in \sigma(\Delta_g)$ be an eigenvalue of multiplicity $m$. Then at least one of the following is true.
\begin{enumerate}[label=(\alph*)]
    \item There exist an open neighborhood $\mathcal U$ of $g$ in $\mathcal G(M)$ and a $C^1$ submersion $\pi: \mathcal U \rightarrow \mathbb R^{\frac{m(m+1)}{2}}$ and $\varepsilon > 0$ such that 
    \begin{align}
        &\sigma(\Delta_{g'}) \cap (\lambda-\varepsilon,\lambda+\varepsilon) = \sigma(\pi(g')),
    \end{align}
    counted with multiplicity, for all $g' \in \mathcal U$.
    \item There exists an orthogonal eigenbasis $(u_j)_{j=1}^m$ of the $\lambda$-eigenspace of $\Delta_g$ and $(\varepsilon_j)_{j=1}^m \in \{-1,0,1\}^m$ such that
    \begin{align}
    \label{equation: linear degeneracy 1}
        \sum_{j=1}^m \epsilon_j u_j^2 = 0,\\
    \label{equation: linear degeneracy 2}
        \sum_{j=1}^m \epsilon_j \, du_j \! \otimes \! d u_j = 0.
    \end{align}
\end{enumerate}
In a fixed conformal class $\mathcal G(M,g)$, the analogous statement holds if one replaces the pair of conditions \eqref{equation: linear degeneracy 1}--\eqref{equation: linear degeneracy 2} by condition \eqref{equation: linear degeneracy 1} alone.
\end{theorem}

Following \cite{CdV1988}, in the first of the two cases we say $(g,\lambda)$ satisfies the Strong Arnold Hypothesis (SAH) with respect to $\mathcal G(M)$ (or $\mathcal G(M,g)$, respectively), and call $\lambda$ a \emph{stable} (resp. \emph{conformally stable}) eigenvalue of $\Delta_g$. In the second case, we call the eigenvalue $\lambda$ \emph{unstable} (resp. \emph{conformally unstable}). 

To prove \Cref{theorem:transverse structure theorem} one first reduces the statement, via perturbation theory, to a pointwise condition on eigenfunctions, as already done implicitly in \cite{CdV1988}. The resulting condition is rather messy in all dimensions greater than $2$, but can be shown to be equivalent to \eqref{equation: linear degeneracy 1}--\eqref{equation: linear degeneracy 2} by a series of reduction steps, which to the best of our knowledge appear only in Besson's paper \cite{Besson1989} (although with important precursors in the work of Bleecker--Wilson and Bando--Urakawa) \footnote{In earlier preprints, this theorem was unfortunately claimed as one of this paper's main results. The authors wish to sincerely apologize for missing Besson's result in their literature search.}.
This simplified condition allows us to prove that the set of metrics with (conformally) unstable eigenvalues is small. In order to make this precise we introduce the following notion.

\begin{definition}
\label{def:tcodim}
A subset $\mathcal{D}$ of a Fr\'echet manifold $\mathcal{X}$ is said to have \tcodim \ $k$ if for every smooth manifold $N$ of dimension less than $k$, the set of maps in $C^\infty (N,\mathcal{X})$ avoiding $\mathcal{D}$ is residual. We say a subset $\mathcal{D}$ of $\mathcal{X}$ has infinite \tcodim \ if it has \tcodim \, $k$ for all $k \in \mathbb N$.
\end{definition}

This notion is meant to capture sets that behave like submanifolds from a transversality theoretic point of view but which might be much more irregular. Indeed, the set of metrics for which the Laplacian has non-simple spectrum has a fairly complicated structure. Using this definition, we can make precise the sense in which metrics admitting a conformally unstable eigenvalue (and therefore also those admitting a unstable eigenvalue) have infinite codimension.

\begin{restatable}{theorem}{unstable}
\label{theorem:nontransverse}
Let $M$ be a closed, connected smooth manifold of dimension at least $2$, and $\mathcal G(M)$ the set of smooth Riemannian metrics on $M$. Then the set of metrics $g \in \mathcal G(M)$ such that $\Delta_g$ has a conformally unstable eigenvalue has infinite \tcodim.
The same holds in any fixed conformal class $\mathcal G(M,g_0)$.
\end{restatable}

\Cref{theorem:transverse structure theorem} and \Cref{theorem:nontransverse} together imply a sweeping statement on the set of metrics admitting eigenvalues of a given multiplicity. It shows that a maximal non-crossing rule indeed holds in our setting.

\begin{restatable}{theorem}{Codimension}
\label{theorem:nonCrossing}
Let $M$ be a closed, connected smooth manifold of dimension at least $2$, and $\mathcal G(M)$ the set of smooth Riemannian metrics on $M$. Then the set of metrics $g \in \mathcal G(M)$ where $\Delta_g$ has an eigenvalue of multiplicity at least $m$ has \tcodim \ $\frac{m(m+1)}{2}-1$.
The same holds in any fixed conformal class $\mathcal G(M,g_0)$.
\end{restatable}

\begin{remark}
    Throughout this paper we work with $C^\infty$ metrics. In previous work on the subject of non-crossing rules (\cite{Uhlenbeck1976},\cite{Teytel1999}), the underlying spaces of metrics had finite regularity and thus a Banach manifold structure. However, all relevant maps occurring in the present paper have finite dimensional domains or codomains, allowing us to carry over the requisite transversality theory. Smooth regularity is crucial for our proof of \Cref{theorem:nontransverse}. We will nevertheless be able to conclude that the maximal non-crossing rule holds for $C^s$ metrics for all $s\geq 0$, see \Cref{corollary: finite regularity}. This corollary is proven from \Cref{theorem:nonCrossing}, using only continuous dependence of the eigenvalues, and holds in lower regularity than previous statements in the literature, which were obtained via perturbation theory in $C^s$-regularity, $s \geq 2$.
\end{remark}

In \cite{BerkolaikoZelenko2024} Berkolaiko and Zelenko analyze topological features of maps $F:N\to \mathbb R^{\frac{m(m+1)}{2}}$ where $N$ is a smooth $k$-dimensional manifold. More specifically, given the $\ell^{th}$ ordered eigenvalue functional $\lambda_\ell$ one may consider the map $$\hat\lambda_\ell :=\lambda_\ell\,\circ \, F : N\to \mathbb R$$ While this function is not smooth, the authors show that $\hat\lambda_\ell$ only has nondegenerate topological critical points provided that the family of symmetric matrices $F$ is chosen so as to be suitably generic, a so-called generalized Morse family. In this setting they prove that the functions $\hat \lambda_\ell$ satisfy Morse inequalities, i.e.\ explicit lower bounds on the number of critical points of a given index in terms of the topology of the underlying parameter space.
The foundational prerequisite for applying these ideas to generic $k$-parameter families of operators such as the Laplace-Beltrami operator is that a transversality theorem like \Cref{theorem:transverse structure theorem} holds for every point and eigenvalue along any generic $k$-parameter family. Thus \Cref{theorem:nontransverse} enables an analysis of the Morse theory of the $\ell^{th}$ ordered eigenvalue functional as performed in \cite{BerkolaikoZelenko2024} in the setting of Laplace-Beltrami operators parametrized by Riemannian metrics.

Given that the (conformal) instability condition is a very strong pointwise condition it is natural to ask whether it can ever be satisfied. The answer to this question is known to be yes \cite{CdV1988}. However, we prove that any unstable eigenvalue must have fairly high multiplicity.
\begin{theorem}
\label{theorem: nonexistence of low multiplicity nontransverse eigenvalues}
    On a connected, closed, smooth Riemannian manifold, there exist no conformally unstable eigenvalues of multiplicity less than four, and no unstable eigenvalues of multiplicity less than seven.
\end{theorem}
There exist examples of a conformally unstable eigenvalue of multiplicity four and of an unstable eigenvalue of multiplicity eight, see Section $4$.\\
To formulate the last result, we fix a metric $g$ and an eigenvalue $\lambda$ of multiplicity $m$. Denote by $\mathcal{V}_{m,k} (M) \subset \mathcal{G}(M)$ (respectively $\mathcal{V}_{m,k}(M,g)\subset\mathcal{G}(M,g)$) the set of smooth metrics on $M$ (respectively smooth metrics on $M$ conformal to $g$) for which the $k^{th}$ eigenvalue has multiplicity $m$.
It is a well known result (see Section $2$) that $\lambda$ being a stable (respectively conformally stable) eigenvalue implies that $\mathcal{V}_{m,k} (M)$ (respectively $\mathcal{V}_{m,k}(M,g)$ is a submanifold of $\mathcal{G}(M)$ (respectively $\mathcal{G}(M,g)$) at $g$. One thus obtains
\begin{theorem}
    Let $M$ be a closed, smooth manifold of dimension at least $2$. The sets $\mathcal{V}_{m,k}(M) \subset \mathcal{G}(M)$ and $\mathcal{V}_{m,k}(M)$ are codimension $\frac{m(m+1)}{2}-1$ submanifolds of $\mathcal{G}(M)$ and $\mathcal{G}(M,g)$ in the complement of sets of transversality codimension infinity, respectively. Furthermore, $\mathcal{V}_{m,k}(M)$ and $\mathcal{V}_{m,k}(M,g)$ are codimension $\frac{m(m+1)}{2}-1$ submanifolds whenever $m\leq 6$ and $m\leq 3$, respectively. Lastly, the set of metrics for which $\mathcal{V}_{m,k}(M,g)$ is not a manifold is in general nonempty.
\end{theorem}
The above theorem is a direct consequence of \Cref{theorem:nontransverse} and \Cref{theorem: nonexistence of low multiplicity nontransverse eigenvalues} combined with \Cref{theorem: general local structure theorem}. A counterexample to $\mathcal{V}_{m,k}(M,g)$ being a manifold is provided in \Cref{ex:nontransverse}.

\subsection{Structure of this paper.} In Section $2$, which is purely expository, we set the stage for the transversality theoretic results: We provide the abstract perturbation-theoretic setup necessary to define the map $\pi$ of \Cref{theorem:transverse structure theorem} and define the Strong Arnold Hypothesis. We then specialize the setup to the case of the Laplace-Beltrami operator and discuss the proof of Besson's Theorem \ref{theorem:transverse structure theorem}. Section $3$ is the heart of the paper, establishing Theorems \ref{theorem:nontransverse} and \ref{theorem:nonCrossing}.
Section $4$ is concerned with a closer examination of unstable eigenvalues, containing a proof that no such eigenvalues exist below multiplicity $7$ (\Cref{theorem: nonexistence of low multiplicity nontransverse eigenvalues}), and examples of (conformally) unstable eigenvalues. We conclude the paper with a number of open questions.

\subsection{Preliminaries}
\label{sec:preliminaries}

In this paper, Einstein's summation convention is used, meaning that any index occurring twice in an expression is automatically summed over. We will often write $\partial_j$ for $\frac{\partial}{\partial x^j}$ in coordinate expressions.

Let $(M,g)$ be a closed, smooth Riemannian manifold. Its Laplace-Beltrami operator $\Delta_g$ is given in coordinates by 
\begin{align*}
    \Delta_g u = |g|^{-\frac{1}{2}} \partial_j (|g|^{\frac{1}{2}} g^{jk} \partial_k u),
\end{align*}
where, as usual, $|g|$ denotes $\det g$ and $g^{jk}$ the components of $g^{-1}$. Of course, this expression arises necessarily if we wish the identity 
\begin{align*}
    \int_M u \Delta_g v d\mu_g = - \int_M g(\nabla_g u , \nabla_g v) \, d\mu_g
\end{align*}
to hold for all $u,v\in C^\infty (M)$, where $\nabla_g$ and $\mu_g$ denote the gradient with respect to $g$ and the volume form, respectively. Let $H^2(M)$ and $L^2(M)$ denote the respective Sobolev spaces, defined with respect to a fixed reference metric on $M$. It is well known that $\Delta_g$ is a bounded operator from $H^2(M)$ to $L^2(M)$. It can be interpreted as a closed unbounded operator on $L^2(M)$, which is self-adjoint with respect to the inner product induced by $g$. Its spectrum consists of one zero eigenvalue and a countable, discrete set of negative eigenvalues with finite multiplicity.

It will be crucial in our proof of \Cref{theorem:nontransverse} that eigenfunctions satisfy a unique continuation theorem: If $M$ is connected, eigenfunctions of $\Delta_g$ cannot vanish to infinite order anywhere. This follows from Aronszajn's unique continuation theorem for general elliptic operators of second order \cite{Aronszajn1957}.

Following Hamilton's book \cite{Hamilton1982}, we call a map $F:\mathcal F \to \mathcal F'$ between Fr\'{e}chet spaces $\mathcal F$ and $\mathcal F'$ \emph{continuously differentiable}, or $C^1$, if 
\begin{align*}
    DF(q)[h] = \lim_{t \to 0} \frac{F(q+th) - F(q)}{t}
\end{align*}
exists everywhere, is linear in $h$, and \emph{jointly continuous} as a function of $q$ and $h$. This notion is stronger than Gateaux differentiability, but much weaker than the definition of continuous differentiability one might expect: The derivative $DF$ need \emph{not} be a continuous map from $\mathcal F$ into the space of continuous linear operators from $\mathcal F$ to $\mathcal F'$. In particular, we never specify a topology on that space.

Let $X,Y$ and $Z$ be Banach spaces, and denote by $\mathcal L(X,Y)$ and $\mathcal L(Y,Z)$ the spaces of continous linear maps from $X$ to $Y$ and from $Y$ to $Z$, respectively, equipped with the operator norm. Suppose $A: \mathcal F \to \mathcal L(Y,Z)$ and $B: \mathcal F \to \mathcal L(X,Y)$ are $C^1$ maps. Then the following version of the product rule holds: For all $q,h \in \mathcal F$,
\begin{align*}
    D(A \circ B)(q)[h] = DA(q)[h] \circ B(q) + A(q) \circ DB(q)[h].
\end{align*}
It is proven in the same way as the product rule in one-variable calculus, since only a directional derivative is taken. The right hand side is clearly jointly continuous in $q$ and $h$, hence $A \circ B: \mathcal F \to \mathcal L(X,Z)$ is $C^1$.

All maps occuring in this paper have either finite dimensional domain or finite dimensional codomain. Thus, none of the usual problems which beset differential geometry in Fr\'{e}chet spaces occur. Borrowing terminology from Freyn \cite{Freyn2015}, we call a subset $\mathcal M \subseteq \mathcal X$ of a Fr\'{e}chet manifold $\mathcal X$ a \emph{cofinite Fr\'{e}chet submanifold} of codimension $N \in \mathbb N$ if, for all $p \in \mathcal M$, there exists a neighborhood $\mathcal U \subseteq \mathcal X$ of $p$ and a $C^1$ submersion $\pi:\mathcal U \to \mathbb R^N$ such that $\mathcal M \cap \mathcal U = \pi^{-1}(0)$.

The standard proof of the implicit function theorem then yields that a cofinite Fr\'{e}chet submanifold is locally a graph over its tangent space at any point $p \in \mathcal M$, and therefore a Fr\'{e}chet manifold modelled on the Fr\'{e}chet space $\ker DF(p) \subseteq T_p \mathcal X$. 
Lacking a reference for this simple observation, we provide a proof in \Cref{section: appendix 0}.

Thom's transversality theory can be carried over with minimal modifications to the setting of smooth maps $f: \mathbb R^K \to \mathcal X$ and cofinite submanifolds $\mathcal M \subseteq \mathcal X$. The necessary statements are provided and proven in \Cref{section: appendix 1}.

The family of operators $\Delta_g: H^2(M) \to L^2(M)$ gives rise to a $C^1$ map 
\begin{align*}
    \Delta_{(\cdot)}: \mathcal G(M) \to \mathcal L(H^2(M), L^2(M)),
\end{align*}
into the Banach space of bounded linear operators from $H^2(M)$ to $L^2(M)$.
Its resolvent $(\lambda - \Delta_g)^{-1}$ is likewise a $C^1$ map from a dense open subset of $\mathbb C \times \mathcal G(M)$ to $\mathcal L(L^2(M), H^2(M))$ (and thus also, if needed, into $\mathcal L(L^2(M))$). Finally, fixing a circle $\gamma$ in $\mathbb C$, the corresponding Riesz projection
\begin{align*}
    P_\gamma(g) = \frac{1}{2\pi i} \int_\gamma (\zeta - \Delta_g)^{-1} d\zeta
\end{align*}
is a $C^1$ map into $\mathcal L(L^2(M))$ on the open subset of metrics admitting no Laplace eigenvalue on $\gamma$. These statements are proven in \Cref{section: appendix 2}.

\section{A local structure theorem for metrics with Laplace eigenvalues of higher multiplicity}
\label{section: local structure theorem}
The aim of this section is to gather results that are known or implicit in the literature and to adapt and recast them in a form appropriate for our purposes. 
In Section \ref{subsection: technical} we recall a well-known setup for perturbation theory for parametric families of good spectral problems. In Section \ref{subsection: specialization to Laplace-Beltrami operator} we will specialize the discussion to the family of Laplace--Beltrami operators parameterized by Riemannian metrics. The main result of \Cref{subsection: specialization to Laplace-Beltrami operator} (Proposition \ref{proposition:submersion equivalent linear nondeg}) has previously been derived by Besson in \cite[Sections 4A  and 7A]{Besson1989}, building on related, though less general, computations by Bleecker and Wilson in \cite[Lemma 5.1.]{BleeckerWilson1980} as well as by Bando and Urakawa in \cite[Lemma 4.4. and Proposition 4.5.]{BandoUrakawa1983}.

\subsection{General setup for the construction of local defining functions and parameter transversality}
\label{subsection: technical}
The construction described here has appeared in various similar forms in the literature (e.g., \cite{CdV1988}, \cite{LupoMicheletti1993}, \cite{Micheletti1998}, \cite{Teytel1999}). In order to prove Theorem \ref{theorem:nontransverse} we will have to use smooth metrics and are thus forced to adapt these methods to a Fr\'{e}chet manifold setting. The requisite implicit function theorem needed for this as well the transversality theoretic results for this setting are presented in Appendix \ref{section: appendix 0} and Appendix \ref{section: appendix 1}, respectively. The setup presented here is furthermore directly applicable to multi-parameter families of operators which are self-adjoint with respect to a variable inner product, without having to find an explicit family of isometries which conjugates the problem into one on a fixed Hilbert space, as in \cite{CdV1988}. At the end of this section, our abstract setup is summarized in \Cref{theorem: general local structure theorem}.

Suppose we are given a separable Fr\'{e}chet manifold $\mathcal{X}$, a $C^1$ family of equivalent inner products $\{\langle-,-\rangle_q\}_{q\in\mathcal{X}}$ on a fixed real Hilbert space $H$, and a family of unbounded operators $\{A_q\}_{q\in \mathcal{X}}$, each self-adjoint with respect to $\langle-,-\rangle_q$ and with compact resolvent. Assume furthermore that the resolvents $(z - A_q)^{-1}$ form a $C^1$ family of bounded operators defined on a dense open subset of $\mathbb C \times \mathcal X$. Now fix a parameter value $q_0$ such that $A_{q_0}$ has an eigenvalue $\lambda$ of multiplicity $m$. Since the eigenvalues of $A$ move continuously with respect to the parameter by the variational characterization of eigenvalues \cite[Theorem 4.10]{GeraldTeschl2014}, there exist $\epsilon >0$ and an open neighbourhood $\mathcal U$ of $q_0$ such that the rank of the spectral projection
\begin{align*}
    P(q)=\frac{1}{2\pi i}\int_\gamma (\zeta - A_q)^{-1} d\zeta, \quad \gamma(t) = \lambda+\epsilon e^{i t} 
\end{align*}
is $m$ for all $q\in \mathcal U$. Define $E_\gamma(q) := \mathrm{Im} \, P(q) $. Consider the family of linear maps $S(q) := P(q_0) \circ P(q) \rvert_{E_\gamma(q_0)}: E_\gamma(q_0) \to E_\gamma(q_0)$. After possibly shrinking $\mathcal U$, $S(q)$ is an isomorphism for all $q \in \mathcal U$, since $P(q)$ varies continuously with $q$, and $P(q_0) \rvert_{E_\gamma(q_0)} = \mathrm{id}\rvert_{E_\gamma(q_0)}$. By construction, $S(q)^{-1} \circ P(q_0): H \to E_\gamma(q_0)$ is a $C^1$ family of left inverses to $P(q) \rvert_{E_\gamma(q_0)}$. Fix $\mu \in \mathbb R$ outside of the spectrum of all $A_q, q \in \mathcal U$, and define the map 
\begin{align*}
    f:\,\mathcal U&\to GL(E_{q_0})\\
     q &\mapsto  S(q)^{-1} \circ P(q_0) \circ \left(\mu - A_q\right)^{-1} \circ \left(P(q)\rvert_{E_\gamma(q_0)}\right).
\end{align*}
The eigenvalues of $f(q)$ are precisely the eigenvalues of $\left(\mu - A_q\right)^{-1}$ corresponding to eigenvalues of $A_q$ encircled by $\gamma$. Thus, $f^{-1}(\mathbb{R}\cdot \mathrm{id})$ contains precisely those $q\in \mathcal U$ for which $A_q$ has an eigenvalue of multiplicity $m$ in $(\lambda-\epsilon,\lambda + \epsilon)$. 

We want to use the map $f$ as a local defining function for the set of parameter values $q\in \mathcal U$ so that $A_q$ has an eigenvalue of multiplicity $m$ that is close to $\lambda$. To achieve this, $f$ has to be modified further to account for the fact that $f(q)$ is symmetric with respect to the inner product $\left<P(q) \, \cdot\, , P(q) \, \cdot \, \right>_q$, which of course depends on $q$. Since this inner product is continuously differentiable in $q$, we may perform Gram-Schmidt orthonormalization to obtain a $C^1$ family of linear isomorphisms $Q(q): \mathbb R^m \to E_{q_0}$ such that $\left<P(q) Q(q) e_k , P(q) Q(q) e_\ell \right>_q = \delta_{k\ell}$. Then, symmetry of $(\mu - A_q)^{-1}$ with respect to $\left< \, \cdot \, , \, \cdot \, \right>_q$ shows that $Q(q)^{-1} \circ f(q) \circ Q(q)$ is a symmetric $m\times m$ matrix, allowing us to define
\begin{align*}
    \pi_\mu:\,\mathcal U & \to \mathbb R^{\frac{m(m+1)}{2}} \\
     q &\mapsto  Q(q)^{-1} \circ f(q) \circ Q(q).
\end{align*}

As a composition of $C^1$ families of linear maps and their inverses (which are $C^1$, e.g, by \cite[Theorem II.3.1.1]{Hamilton1982}), $\pi_\mu$ is $C^1$. Denote $R(q) = (\mu - A_{q})^{-1}$ and $\tilde P(q) = P(q)\rvert_{E_\gamma(q_0)}$. Given $h\in T_{q_0} \mathcal{X}$, we first calculate
\begin{align*}
    Df (q_0) [h]= 
    &- S(q_0)^{-1} \circ DS(q_0)[h] \circ S(q_0)^{-1} \circ P(q_0) \circ R(q_0) \circ \tilde P(q_0) \\ &+ S(q_0)^{-1} \circ P(q_0) \circ DR(q_0)[h] \circ \tilde P(q_0) \\ &+ S(q_0)^{-1} \circ P(q_0) \circ R(q_0) \circ D \tilde P(q_0)[h] \\
    & = - \tfrac{1}{\mu - \lambda} DS(q_0)[h] + P(q_0) \circ DR(q_0)[h] +  P(q_0) \circ R(q_0) \circ D \tilde P(q_0)[h],
\end{align*}
where we used $S(q_0) = \mathrm{id}$ and that $ R(q_0)$ acts on $E_{\gamma}(q_0)$ via multiplication by $\frac{1}{\mu - \lambda}$. Furthermore, $P(q_0) \circ R(q_0) = \frac{1}{\mu - \lambda} P(q_0)$. This allows us to further simplify
\begin{align*}
    Df (q_0) [h] &= 
     P(q_0) \circ DR(q_0)[h] - \tfrac{1}{\mu - \lambda} \left( DS(q_0)[h] - P(q_0) \circ D \tilde P(q_0)[h] \right) \\
     &= P(q_0) \circ DR(q_0)[h],
\end{align*}
because the term in brackets vanishes by the product rule. Now the derivative of $\pi_\mu$ at $q_0$ is computed to be
\begin{align*}
    D\pi_\mu (q_0) [h] &= 
     - Q(q_0)^{-1} \circ DQ(q_0)[h] \circ Q(q_0)^{-1} \circ f(q_0) \circ Q(q_0) \\
     & + Q(q_0)^{-1} \circ Df(q_0)[h] \circ Q(q_0) + Q(q_0)^{-1} \circ f(q_0) \circ DQ(q_0)[h] \\
     &= Q(q_0)^{-1} \circ Df(q_0)[h] \circ Q(q_0) = Q(q_0)^{-1} \circ P(q_0) \circ DR(q_0)[h] \circ Q(q_0),
\end{align*}
since $f(q_0) = \frac{\mathrm{id}}{\mu-\lambda}$. To obtain the entries of the matrix $D\pi_\mu (q_0) [h]$, let $\{e_k\}_{k=1}^m$ denote the standard basis of $\mathbb R^m$, so that $\big\{Q(q_0)e_k\big\}_{k=1}^m$ is an orthonormal basis of the eigenspace $E_\gamma(q_0)$. For $1 \leq k,\ell \leq m$, we have
\begin{align*}
    \left\langle D\pi_\mu (q_0) [h] e_k, e_\ell \right\rangle_{\mathbb R^m} &=
    \left\langle P(q_0) \circ DR(q_0)[h] \circ Q(q_0) e_k , Q(q_0) e_\ell \right\rangle_{q_0} \\
    &= \left\langle DR(q_0)[h] \circ Q(q_0) e_k , Q(q_0) e_\ell \right\rangle_{q_0}.
\end{align*}
It remains to remove the dependence on the arbitrary choice of resolvent. Note that, for all $q \in \mathcal U$, $A_q \rvert_{E_\gamma(q)} = \mu - R(q)^{-1}$. Thus, if we define $\pi(q) = \mu - \pi_\mu(q)^{-1}$, then the spectrum of $A_q$ encircled by $\gamma$ agrees with that of the matrix $\pi(q)$. Since $\pi_\mu(q_0) = \frac{\mathrm{id}}{\mu - \lambda}$, the quotient rule implies
\begin{align*}
    \left\langle D\pi (q_0) [h] e_k, e_\ell \right\rangle_{\mathbb R^m} &= \left\langle (\mu - \lambda)^2 DR(q_0)[h] \circ Q(q_0) e_k , Q(q_0) e_\ell \right\rangle_{q_0}.
\end{align*}
This matrix is, by construction, independent of the choice of $\mu$. Alternatively, one may also differentiate the first resolvent identity to show that the operator $(\mu - \lambda)^2 P(q_0) \circ DR(q_0)[h] \rvert_{E_\gamma(q_0)}$ is independent of this choice.

In the case that all operators $A_q$ share the same domain (which will be true for the Laplace-Beltrami operators on a closed Riemannian manifold $M$), one can arrive at the expression $\left\langle (\mu - \lambda)^2 DR(q_0)[h] \circ Q(q_0) e_k , Q(q_0) e_\ell \right\rangle_{q_0}$ directly, without considering the resolvent.
More precisely, suppose there is a fixed Banach space $D$ which embeds continuously into $H$, and $A: \mathcal X \to \mathcal L(D,H)$ is a $C^1$ family of bounded linear operators. Differentiating the identity $R(q)\circ (\mu-A_q) = \mathrm{id}_{D}$ yields
\begin{align*}
    (\mu-\lambda) DR(q_0)[h]\rvert_{E_\gamma(q_0)} = - R(q_0) DA(q_0)[h] \rvert_{E_\gamma(q_0)},
\end{align*}
from which, since $P(q_0) \circ R(q_0) = (\mu-\lambda)^{-1} P(q_0)$, it follows that
\begin{align*}
    (\mu-\lambda)^2 P(q_0) \circ DR(q_0)[h]\rvert_{E_\gamma(q_0)} = - P(q_0) \circ DA(q_0)[h] \rvert_{E_\gamma(q_0)},
\end{align*}
If we denote the orthonormal basis $\{Q(q_0) e_k\}_{k=1}^m$ by $\{u_k\}_{k=1}^m$, we can thus write $D\pi(q_0)$ in a very convenient way:
\begin{align}
    D\pi(q_0)[h]_{k\ell} = \left\langle DA(q_0)[h] u_k , u_\ell \right\rangle_{q_0}
\end{align}
It is this expression that we will work with in the following sections. 

Surjectivity of $D\pi(q)$ is an open condition, since $\pi$ has a finite dimensional codomain. Thus, if $D\pi(q_0)$ is surjective, we can find an open neighborhood $\mathcal U$ of $q_0$ such that $\pi \rvert_\mathcal U$ is a submersion. In particular, as mentioned in \Cref{sec:preliminaries} (Preliminaries), this means that the set of parameters $q$ such that that the eigenvalue $\lambda$ does not split up is a cofinite Fr\'{e}chet submanifold, because it is given as the preimage $\pi^{-1}(\mathbb R \cdot \mathrm{id})$, i.e., the preimage of a manifold under a submersion.

We collect the results of the preceding discussion in the following

\begin{proposition}
\label{theorem: general local structure theorem}
    Let $\mathcal{X}$ be a Fr\'{e}chet manifold, $H$ a Hilbert space, $\left\{\langle-,-\rangle_q\right\}_{q\in \mathcal{X}}$ a $C^1$ family of inner products on $H$, and $\{A_q\}_{q\in \mathcal{X}}$ a family of unbounded operators, each self-adjoint with respect to $\langle-,-\rangle_q$ and with compact resolvent. Assume that the resolvent $(z-A_q)^{-1}$ is a $C^1$ family of bounded operators defined on a dense open subset of $\mathbb C \times \mathcal X$.
    
    Suppose $\lambda$ is an eigenvalue with multiplicity $m$ of $A_{q_0}$ for some $q_0 \in \mathcal X$. Denote $R(q) = (\mu-A_q)^{-1}$, for some $\mu \in \rho(A_{q_0}) \cap \mathbb R$, and choose an orthonormal basis $\{u_k\}_{k=1}^m$ of the $\lambda$-eigenspace of $A_{q_0}$. If the linear map
    \begin{align*}
        h \mapsto \big(\left\langle DR(q_0)[h] u_k, u_\ell \right\rangle\big)_{k,\ell = 1}^m
    \end{align*}
    from $T_{q_0} \mathcal X$ to $\mathbb R^{\frac{m(m+1)}{2}}$ is surjective, there exist an open neighborhood $\mathcal U$ of $q_0$ in $\mathcal{X}$, a $C^1$ submersion $\pi: \mathcal U \rightarrow \mathbb R^{\frac{m(m+1)}{2}}$, and $\varepsilon > 0$, such that 
    \begin{align}
        &\sigma(A_{q}) \cap (\lambda-\varepsilon,\lambda+\varepsilon) = \sigma(\pi(q)),
    \end{align}
    counted with multiplicity, for all $q \in \mathcal U$. If, additionally, all operators $A_q$ share the same domain $D$, the linear map 
    \begin{align*}
        h \mapsto  \big( \left\langle DA(q_0)[h] u_k, u_\ell \right\rangle \big)_{k,\ell = 1}^m
    \end{align*}
    is well-defined and it is sufficient to check its surjectivity instead.
    
    The preimage $\pi^{-1}(\mathbb{R}\cdot \mathrm{id})$ is a Fr\'{e}chet submanifold of codimension $\frac{m(m+1)}{2}-1$.
\end{proposition}

In view of this proposition we can formulate the strong Arnold hypothesis introduced by Colin de Verdi\`ere \cite{CdV1988} very conveniently.

\begin{definition}
    The family of operators $\{A_q\}_{q\in \mathcal{X}}$ is said to satisfy the strong Arnold hypothesis (SAH) if $\pi$ is a submersion for all $q$ and $\lambda$. We say it satisfies the SAH at $q$ and $\lambda$ if $\pi$ is a submersion at $q$ for the eigenvalue $\lambda$.
\end{definition}

If the family $A_q$ satisfies the strong Arnold hypothesis, one can prove a non-crossing rule akin to \Cref{theorem:nonCrossing} in this very abstract setting provided that $\mathcal X$ is separable. This hypothesis is needed in several places to extract countable open covers, using the fact that a separable metric space is Lindelöf (i.e, every open cover admits a countable subcover). The family of Laplace-Beltrami operators does not satisfy the SAH, and so a proof of \Cref{theorem:nonCrossing} must wait until the end of \Cref{section: nontransverse eigenvalues}. Assuming the SAH, the proof of \Cref{theorem:nonCrossing} can be carried over to the abstract setting of \Cref{theorem: general local structure theorem}. We omit this to avoid overloading this article.

\subsection{The case of the Laplace-Beltrami operator: Proof of Theorem \ref{theorem:transverse structure theorem}}\label{subsection: specialization to Laplace-Beltrami operator}
We now specialize the results of the preceding section to $\mathcal{X}=\mathcal{G} (M)$, the space of smooth metrics on a closed manifold $M$, or a fixed conformal class $\mathcal{X}=\mathcal{G} (M,g_0)$, and $A_q=\Delta_g$, the Laplace-Beltrami operator. We stress that the following Proposition was proven by Besson in \cite{Besson1989}, for the conformal case in \cite[Section 4A]{Besson1989} and for the general case in \cite[Section 7A]{Besson1989}. We supply its proof both for the convenience of the reader and in order to streamline the literature on this subject.

\begin{proposition}\cite[Sections 4A and 7A]{Besson1989}
\label{proposition:submersion equivalent linear nondeg}
    Let $M$ be a closed smooth manifold of dimension $n \geq 2$. The following are equivalent:
    \begin{enumerate}
        \item $\{\Delta_g\}_{g\in \mathcal{G}(M)}$ does not satisfy the SAH at $g$ and $\lambda$. \label{statement 1}
        \item There exist a basis $\{u_\alpha\}$ of $E_\lambda$ and $A \in \mathbb R^{\frac{m(m+1)}{2}}\setminus\{0\}$ such that
    \begin{align}
    \label{equation: localization}
        A^{\alpha \beta}\bigg(\frac{\lambda}{2}u_\alpha u_\beta \, g+\frac{1}{2} g(\nabla u_\alpha,\nabla u_\beta) \, g - du_\alpha \otimes du_\beta \bigg) \equiv 0.
    \end{align}\label{statement 2}
        \item There exist a basis $\{u_\alpha\}$ of $E_\lambda$ and $A\in \mathbb R^{\frac{m(m+1)}{2}}\setminus\{0\}$ such that
        \begin{align*}
            A^{\alpha\beta}u_\alpha u_\beta \equiv 0
        \end{align*}
        and 
        \begin{align*}
            A^{\alpha\beta} du_\alpha \otimes du_\beta \equiv 0.
        \end{align*}\label{statement 3}
        \item $\lambda$ is an unstable eigenvalue of $\Delta_g$. \label{statement 4}
    \end{enumerate}
Similarly, in the conformal setting, the following are equivalent:
    \begin{enumerate}
        \item $\{\Delta_g\}_{g\in \mathcal{G}(M,g_0)}$ does not satisfy the SAH at $g$ and $\lambda$.
        \item There exist a basis $\{u_\alpha\}$ of $E_\lambda$ and $A \in \mathbb R^{\frac{m(m+1)}{2}} \setminus\{0\}$ such that
    \begin{align*}
    A^{\alpha\beta}\bigg(\lambda u_\alpha u_\beta + \frac{n-2}{n} g(\nabla u_\alpha, \nabla u_\beta) \bigg)\equiv 0.
    \end{align*}
        \item There exist a basis $\{u_\alpha\}$ of $E_\lambda$ and $A\in \mathbb R^{\frac{m(m+1)}{2}} \setminus\{0\}$ such that
        \begin{align*}
            A^{\alpha\beta}u_\alpha u_\beta \equiv 0.
        \end{align*}
        \item $\lambda$ is a conformally unstable eigenvalue of $\Delta_g$.
    \end{enumerate}
\end{proposition}
\begin{proof}
    $\textbf{(\ref{statement 1})}\iff\textbf{(\ref{statement 2})}$
We express the Laplace operator in local coordinates:
\begin{align*}
    \Delta_g u=|g|^{-\frac{1}{2}}\partial_j (|g|^{\frac{1}{2}}g^{jk}\partial_k u)
\end{align*}
where by $|g|$ we mean the determinant of $g$. Given a symmetric $(0,2)$-tensor field $h$ and using the identities
\begin{align*}
    D_h |g|^s=s\cdot tr_g (h) |g|^s,\quad D_h g^{ij}=-h^{ij}, 
\end{align*}
one can compute the variation of $\Delta_g$ in direction $h$:
\begin{align*}
    D_h\Delta_g u = &-\frac{1}{2}tr_g (h)\Delta_g u + \frac{1}{2}|g|^{-\frac{1}{2}}\partial_j (tr_g (h)|g|^{\frac{1}{2}}g^{jk}\partial_k u)
    - |g|^{-\frac{1}{2}}\partial_j (|g|^{\frac{1}{2}}h^{jk}\partial_k u)\\
    = & \frac{1}{2}\partial_j (tr_g (h)) g^{jk}\partial_k u - |g|^{-\frac{1}{2}}\partial_j (|g|^{\frac{1}{2}}h^{jk}\partial_k u)
\end{align*}
We now compute the $L^2$ inner product
\begin{align*}
    \langle u, D_h\Delta_g v\rangle = & \frac{1}{2}\int_M \partial_j (tr_g (h)) g^{jk}(\partial_k u) v\, d\mu_g -\int_M |g|^{-\frac{1}{2}}\partial_j (|g|^{\frac{1}{2}}h^{jk}\partial_k u) v \,d\mu_g \\
    = & - \frac{1}{2} \int_M \left( tr_g (h)(\Delta_g u) v + tr_g (h)g(\nabla u,\nabla v) - 2 h^{jk}[du \otimes dv]_{jk} \right) \,d\mu_g \\
    = & \int_M h^{jk} \left( -\frac{1}{2} g_{jk} (\Delta_g u) v -\frac{1}{2} g_{jk} g(\nabla u,\nabla v) + [du \otimes dv]_{jk} \right) \,d\mu_g \\
\end{align*}

Pick an arbitrary basis $\{u_\alpha\}$ of $E_\lambda$. If the SAH is not satisfied at $g$, then \Cref{theorem: general local structure theorem} implies there exists a nonzero symmetric matrix $A \in \mathbb R^{\frac{m(m+1)}{2}}$ such that for all $h$, $A^{\alpha\beta}\langle u_\alpha,D_h\Delta_g u_\beta\rangle = 0$. We obtain
\begin{align*}
    0 = & \int_M h^{jk} A^{\alpha\beta}\left( \frac{\lambda}{2} g_{jk} u_\alpha u_\beta +\frac{1}{2} g_{jk} g(\nabla u_\alpha,\nabla u_\beta) - [du_\alpha \otimes du_\beta]_{jk} \right) \,d\mu_g 
\end{align*}
for all symmetric $(0,2)$-tensor fields $h$. Since the expression in brackets is itself a symmetric $(0,2)$-tensor field, this is equivalent to
\begin{align*}
    0 \equiv A^{\alpha\beta}\left( \frac{\lambda}{2} u_\alpha u_\beta \, g +\frac{1}{2} g(\nabla u_\alpha,\nabla u_\beta) \, g - du_\alpha \otimes du_\beta \right).
\end{align*}
\par

$\textbf{(\ref{statement 2})}\iff\textbf{(\ref{statement 3})}$
The converse direction is clear. For the other direction, we contract equation \ref{equation: localization} with the metric $g$ 
    \begin{align}\label{equation: contracted degeneracy}
        A^{\alpha\beta}\bigg(\frac{n\lambda}{2}u_\alpha u_\beta + \frac{n-2}{2} g(\nabla u_\alpha, \nabla u_\beta) \bigg)=0
    \end{align}
In dimension $n=2$, this implies that $A^{\alpha\beta}u_\alpha u_\beta=0$. Applying the Laplacian to this identity yields
\begin{align*}
    0=\Delta_g A^{\alpha\beta}u_\alpha u_\beta=2 A^{\alpha\beta}u_\alpha \Delta_g u_\beta + 2 A^{\alpha\beta} g(\nabla u_\alpha,\nabla u_\beta)=2 A^{\alpha\beta} g(\nabla u_\alpha,\nabla u_\beta)
\end{align*}
Reinserting these two expressions into equation \ref{equation: localization} gives the desired conclusion.\\
In dimension $n\geq 3$, consider the function $F=A^{\alpha\beta}u_\alpha u_\beta$ and apply the Laplacian
\begin{align*}
    \Delta_g F&=2 A^{\alpha\beta}u_\alpha \Delta_g (u_\beta) + 2 A^{\alpha\beta} g(\nabla u_\alpha,\nabla u_\beta)\\
    &= (2\lambda-2\lambda \frac{n}{n-2})A^{\alpha\beta}u_\alpha u_\beta\\
    &=2\lambda \big(1-\frac{n}{n-2}\big) F
\end{align*}
where we used equation \ref{equation: contracted degeneracy}. Note that the coefficient of $F$ on the right is positive. It follows that $F\equiv 0$. \par

$\textbf{(\ref{statement 3})}\iff\textbf{(\ref{statement 4})}$

The matrix $A$ is symmetric, and hence, it can be diagonalized. Let $B^\gamma_\delta$ be a change of basis such that $A^{\alpha\beta}=B^\alpha_\gamma D^{\gamma\delta}B^\beta_\delta$, where $D$ is diagonal. Then
\begin{align*}
    A^{\alpha\beta}u_\alpha u_\beta = B^\alpha_\gamma D^{\gamma\delta}B^\beta_\delta u_\alpha u_\beta =D^{\gamma\delta}v_\gamma v_\delta
\end{align*}
where $v_\gamma = B^\alpha_\gamma u_\alpha$. Upon rescaling the basis vectors $\{v_\gamma\}$ we may assume that the diagonal entries of $D$ are valued in $\{-1,0,1\}$. The basis $\{v_\alpha\}$ may be chosen to be orthogonal: One may choose the original basis $\{u_\alpha\}$ orthonormal, and diagonalize the symmetric matrix $A^{\alpha\beta}$ by an orthogonal change of basis.
\par
The proof for the conformal case is very similar. We can again use integration by parts to show that $\pi$ is a submersion if and only if there is no basis $\{u_\alpha \}$ and $A\in \mathbb R^{\frac{m(m+1)}{2}}$ so that the traced version of equation \ref{equation: localization} holds. Proving the equivalence of $(2)$ and $(3)$ again works by applying the Laplacian to $A^{\alpha\beta}u_\alpha u_\beta$. 
\end{proof}

\begin{remark}
    We would like to point out that a version of the first equivalence in Proposition \ref{proposition:submersion equivalent linear nondeg} always works for self-adjoint differential operators on some $L^2$ as long as the space of variations is closed under multiplication with bump functions. One can use integration by parts to isolate $h$ in $D_h A$ and then localize. The resulting expression may be unwieldy (this is the case, e.g., for the Hodge Laplacian).

A particularly interesting class of examples is that of conformally covariant operators on vector bundles equipped with a metric $g$ as treated by Canzani \cite{Canzani2014}. An inspection of the conformal variation computed in \cite[Proposition 5.4]{Canzani2014} allows us to conclude that the failure of the SAH at a metric $g$ for an eigenvalue $\lambda$ in the conformal class is equivalent to the following:
            There exists a basis $\{u_\alpha\}$ of $E_\lambda$ and an element $A\in \mathbb R^{\frac{m(m+1)}{2}}$ so that
        \begin{align*}
            A^{\alpha\beta} g(u_\alpha , u_\beta)  \equiv 0.
        \end{align*}
Specializing to conformally covariant operators acting on smooth functions, this becomes
        \begin{align*}
            A^{\alpha\beta} u_\alpha  u_\beta  \equiv 0.
        \end{align*}
just like for the Laplace-Beltrami operator. Canzani proved that this condition is never satisfied in multiplicity $2$, i.e., that there are no conformally unstable eigenvalues of multiplicity $2$ in this setting \cite[Proposition 5.6]{Canzani2014}. Actually her proof shows something stronger, namely that even in higher multiplicity there can never be a pair of eigenfunctions $u_1$ and $u_2$ such that the above identity is satisfied. Inspecting the proof of Proposition \ref{theorem: general local structure theorem} it is not hard to see that this implies a $1$-parameter non-crossing rule for all conformally covariant operators acting on smooth functions. 

The preceding observations and the proof of \Cref{theorem:nontransverse} (\Cref{section: nontransverse eigenvalues}), which develops techniques for showing that relations between eigenfunctions of the form $A^{\alpha\beta} u_\alpha  u_\beta  \equiv 0$ are rarely satisfied, suggest that our methods can be applied to a great variety of operators.
\end{remark}

\Cref{theorem:transverse structure theorem} follows directly by combining \Cref{theorem: general local structure theorem} and \Cref{proposition:submersion equivalent linear nondeg}.

Theorem \ref{theorem:transverse structure theorem}, together with the cofinite transversality theory presented in \Cref{section: appendix 1}, already implies that Theorem \ref{theorem:nonCrossing} holds in the complement of those metrics admitting an unstable eigenvalue. In order to establish Theorem \ref{theorem:nonCrossing} in full we still have to show that the set of such metrics is small. A careful proof of \Cref{theorem:nonCrossing} is therefore presented at the end of \cref{section: nontransverse eigenvalues}.

Before proceeding with our analysis of unstable eigenvalues, we show that they are in fact fully characterized by the condition $\sum_{j=1}^k \epsilon_j \, du_j \! \otimes \! du_j = 0$ (\Cref{equation: linear degeneracy 2}), since this condition already implies $\sum_{j=1}^k \epsilon_j  u_j^2 = 0$ (\Cref{equation: linear degeneracy 1}). This fact will not be used in the rest of the text.

\begin{proposition}
    \label{proposition: further reduction}
    Let $(M,g)$ be a closed Riemannian manifold. Suppose there exist $m$ eigenfunctions $\{u_\alpha\}_{\alpha=1}^m$ of $\Delta_g$, corresponding to an eigenvalue $\lambda < 0$, and a symmetric matrix $A^{\alpha\beta} \in \mathbb R^{\frac{m(m+1)}{2}}$, such that
    \begin{align*}
        A^{\alpha\beta} \, du_\alpha \! \otimes \! du_\beta = 0.
    \end{align*}
    Then $A^{\alpha\beta} u_\alpha u_\beta = 0$ as well.
\end{proposition}

\begin{proof}
    In coordinates, $A^{\alpha\beta} \, du_\alpha \! \otimes \! du_\beta = 0$ corresponds to the identity
    \begin{align*}
        A^{\alpha\beta} \partial_k u_\alpha \partial_\ell u_\beta = 0,
    \end{align*}
    for all $k,\ell = 1,\ldots,n$. Taking the derivative of this expression in the $\partial_j$-direction and contracting with $g^{k\ell}$ yields 
    \begin{align*}
        A^{\alpha\beta} g^{k\ell}\partial_j\partial_k u_\alpha \partial_\ell u_\beta + A^{\alpha\beta} g^{k\ell}\partial_k u_\alpha \partial_j\partial_\ell u_\beta = 0.
    \end{align*}
    On the other hand, because both $A$ and $g$ are symmetric,
    \begin{align*}
        &A^{\alpha\beta} g^{k\ell}\partial_j\partial_k u_\alpha \partial_\ell u_\beta =
        A^{\beta\alpha} g^{\ell k}\partial_j\partial_\ell u_\beta \partial_k u_\alpha =
        A^{\alpha\beta} g^{k\ell} \partial_k u_\alpha \partial_j\partial_\ell u_\beta,
    \end{align*}
    so $A^{\alpha\beta} g^{k\ell}\partial_j\partial_k u_\alpha \partial_\ell u_\beta = 0$. This allows us to reduce the expression which results from taking the divergence of $0 = A^{\alpha\beta} \, \nabla u_\alpha \! \otimes \! \nabla u_\beta$ in the first slot. In coordinates, this amounts to the following calculation, valid for all $\ell = 1,\ldots,n$:
    \begin{align*}
        0 &= A^{\alpha\beta} |g|^{- \frac{1}{2}} \partial_j \left( |g|^{\frac12} g^{jk} \partial_k u_\alpha \partial_\ell u_\beta \right) \\
        &= A^{\alpha\beta} \left( (\Delta_g u_\alpha) \partial_\ell u_\beta + g^{jk} \partial_k u_\alpha \partial_j\partial_\ell u_\beta \right) = 
        \lambda A^{\alpha\beta} u_\alpha \partial_\ell u_\beta
    \end{align*}
    Now, we take the divergence in the second slot and obtain
    \begin{align*}
        0 &= A^{\alpha\beta} |g|^{- \frac{1}{2}} \partial_j \left( u_\alpha |g|^{\frac12} g^{j\ell} \partial_\ell u_\beta \right) \\
        &= A^{\alpha\beta} \left( u_\alpha \Delta u_\beta + g^{j\ell} \partial_j u_\alpha \partial_\ell u_\beta \right) = \lambda A^{\alpha\beta} u_\alpha u_\beta,
    \end{align*}
    as desired.
\end{proof}

\section{Metrics with a (conformally) unstable eigenvalue are rare}
\label{section: nontransverse eigenvalues}

By the results of the preceding section, the behavior of stable multiple eigenvalues under perturbation of the metric parallels that of multiple eigenvalues of symmetric matrices. To get an unconditional result like \Cref{theorem:nonCrossing}, one must still prove that metrics with an unstable eigenvalue are rare. Examples of metrics with an unstable eigenvalue indeed exist, as the examples of Colin de Verdi\`ere (discussed in \Cref{section:low multiplicities and examples}) show. The aim of this section is to prove \Cref{theorem:nontransverse}, and thus establish \Cref{theorem:nonCrossing}.

The proof proceeds as follows: We first stratify the set of metrics admitting conformally unstable eigenvalues into several pieces which can be covered locally by cofinite Fr\'{e}chet submanifolds of arbitrarily high codimension. The key distinction turns out to be whether the fundamental solution to the Helmholtz equation $(\lambda - \Delta) u = \delta_p$ is a quotient of smooth functions whose denominator does not vanish to infinite order anywhere. Those strata where this is not the case are shown to be of infinite \tcodim \ using basic microlocal analysis (\Cref{prop:generalizednontransverseMetric}), and the other case is dealt with using heat kernel asymptotics (\Cref{lemma: smooth kernel}). A similar argument was employed in \cite{Zelditch1990} in the weaker context of showing that a related differential is nonzero, but seemingly resting on an erroneous analysis of the Green's function's asymptotics at the diagonal. The section concludes with a careful proof of \Cref{theorem:nonCrossing}, which at the same time provides the final arguments necessary to establish \Cref{theorem:nontransverse}.

It is convenient to introduce the ring of quotients of smooth functions whose denominator does not vanish to infinite order anywhere. We shall denote it by $\mathcal Q^\infty(M)$, and sometimes write $\frac{\varphi}{\psi} \in \mathcal Q^\infty(M)$ for its elements, where the denominator $\psi \in C^\infty(M)$ has finite vanishing order everywhere.

We begin by stratifying the set $\mathcal S \subseteq \mathcal G(M)$ of metrics admitting a conformally unstable eigenvalue.
\begin{enumerate}[label=(\roman*)]
    \item Let $\mathcal S_\infty(\lambda)$ denote those metrics for which $\lambda \in \sigma(\Delta_g)$, and the Schwartz kernel of the operator $L$ defined by $L(\lambda - \Delta_g) = (\lambda - \Delta_g)L = \mathbb I - P_\lambda$ is an element of $\mathcal Q^\infty(M\times M)$.
    \item For $m \in \mathbb N$, $k \in \{1,\ldots,\frac{m(m+1)}{2}\}$ and $\lambda \in (-\infty,0)$, let $\mathcal S_{m,k}(\lambda)$ denote the set of $g \in \mathcal G(M) \setminus \mathcal S_\infty(\lambda)$ such that $\lambda \in \sigma(\Delta_g)$ has multiplicity $m$, and, given a basis $\{u_\alpha\}_{\alpha=1}^m$ of $E_\lambda$, the equation $A^{\alpha\beta}u_\alpha(x) u_\beta(x) \equiv 0$ has exactly $k$ linearly independent solutions $A \in \mathbb R^{\frac{m(m+1)}{2}}$.
\end{enumerate}

The next proposition shows that the union of strata $\mathcal S_{m,k}(\lambda)$ can be covered locally by Fr\'{e}chet manifolds of finite, but arbitrarily high, codimension, if $\lambda$ is constrained in a suitably small interval.

\begin{proposition}
    \label{prop:generalizednontransverseMetric}
    Let $M$ be a closed, connected smooth manifold of dimension at least $2$. Consider a metric $g \in \mathcal S_{m,k}(\lambda)$, for some $m, k \in \mathbb N$, $\lambda \in (-\infty,0)$. Let $\gamma: \mathbb S^1 \rightarrow \mathbb C$ be a smooth, positively oriented Jordan curve encircling $\lambda$, but no other point of $\sigma(\Delta_g)$. 
    Let $\mathcal U$ be a neighborhood of $g$ in $\mathcal G(M)$ such that for each $\tilde g \in \mathcal U$, exactly $m$ eigenvalues of $\Delta_{\tilde g}$, counted with multiplicity, are encircled by $\gamma$. For $\tilde g \in \mathcal U$, denote the sum of the eigenspaces of these $m$ eigenvalues by $E_\gamma(\tilde g)$.
    
    Let $\widetilde{\mathcal S}_{m,k} \subseteq \mathcal U$ denote the set of all $\tilde g \in \mathcal U$ for which there exists a basis $\{\tilde u_\alpha\}_{\alpha=1}^m$  of $E_\gamma(\tilde g)$ and at least $k$ linearly independent matrices $\tilde A \in \mathbb R^{\frac{m(m+1)}{2}}$ with $\tilde A^{\alpha\beta} \tilde u_\alpha \tilde u_\beta = 0$. Then, for any $N \in \mathbb N$, the set $\widetilde{\mathcal S}_{m,k}$ is contained in a Fr\'{e}chet submanifold of $\mathcal G(M)$ of codimension $N$, after possibly shrinking $\mathcal U$. An identical statement also holds in any conformal class $\mathcal G(M,g_0)$.
\end{proposition}

Before commencing with the proof of \Cref{prop:generalizednontransverseMetric}, it is important to observe that the condition defining $\widetilde{\mathcal S}_{m,k}$ is independent of the choice of basis for $E_\gamma(\tilde g)$, since changing to a different basis just amounts to conjugating the matrix $\tilde A$. The key point of the above, somewhat technical, definition is that $\widetilde{\mathcal S}_{m,k}$ can be analyzed via the implicit function theorem, and $\mathcal S_{m,k}(\tilde \lambda) \cap \mathcal U \subseteq \widetilde{\mathcal S}_{m,k}$ for all $\tilde \lambda$ encircled by $\gamma$.

\begin{lemma}
\label{lemma:nontransverse defining function}
    \begin{enumerate}
        \item Possibly after shrinking $\mathcal U$, one continuously differentiable choice $\left(\tilde u_\alpha\right)_{\alpha = 1}^m$ of an eigenbasis for $E_\gamma(\tilde g)$ is given by $\tilde u_\alpha = P(\tilde g) u_\alpha$, where $P(\tilde g)$ denotes the spectral projection onto $E_\gamma(\tilde g)$.
        \item Define $F:\mathcal U \times \mathbb R^{\frac{m(m+1)}{2}} \rightarrow C^\infty(M)$ by
        \begin{align*}
            &F(\tilde g, \tilde A) =  \tilde A^{\alpha\beta}  P_\gamma(\tilde g) u_\alpha P_\gamma(\tilde g) u_\beta.
        \end{align*}
        Then $F$ is continuously differentiable. Its derivative at $(g,A)$ with respect to the metric is explicitly given by
        \begin{align*}
            &D_{\tilde g} F(g, A)[h] = 2 A^{\alpha\beta} (L\Delta'_h u_\alpha) u_\beta,
        \end{align*}
        where $\Delta'_h$ is shorthand for $D\Delta(g)[h]$ and $L: L^2(M) \rightarrow H^2(M)$ is the unique operator satisfying $(\lambda - \Delta_g) \circ L = \mathbb I - P_\lambda(g)$ and $L \circ (\lambda - \Delta_g) = \mathbb I - P_\lambda(g)$ on $L^2(M)$ and $H^2(M)$, respectively.
        \item Let $V \subseteq \mathbb R^{\frac{m(m+1)}{2}}$ be a subspace of codimension $k-1$ intersecting $\ker F(g,\cdot)$ transversally. Let $\mathcal A := \{ A \in V: \tr(A^TA) = 1\}$, and $F_1 = F\rvert_{\mathcal U \times \mathcal A}$. Then $\widetilde{\mathcal S}_{m,k} \subseteq \pi(F_1^{-1}(\{0\}))$, where $\pi: \mathcal U \times \mathcal A \to \mathcal U$ denotes the natural projection.
    \end{enumerate}
\end{lemma}

\begin{proof}
    We begin by differentiating under the contour integral for the spectral projection, justified by the dominated convergence theorem for Bochner integrals.
    \begin{align*}
        D_{g} P_\gamma &= D_{g} \left( \frac{1}{2\pi i} \int_\gamma (z - \Delta)^{-1} dz \right) = \frac{1}{2\pi i} \int_\gamma (z - \Delta)^{-1} \Delta' (z - \Delta)^{-1} dz.
    \end{align*}
    If $\Delta_{g} u = \lambda u$, we have
    \begin{align*}
        D_{g} P_\gamma(g) u &= \frac{1}{2\pi i} \left( \int_\gamma (z-\lambda)^{-1} (z - \Delta)^{-1} dz \right) \Delta' u.
    \end{align*}
    By the Riesz-Dunford calculus, $L = \frac{1}{2\pi i} \int_\gamma (z-\lambda)^{-1} (z - \Delta)^{-1} dz$ is the unique operator satisfying $L(\lambda - \Delta_g) = (\lambda - \Delta_g)L = \mathbb{I} - P_\lambda(g)$. The map $\tilde u_\alpha = P(\tilde g) u_\alpha$ is thus continuously differentiable, with derivative $D \tilde u_\alpha (g)[h] = L \Delta'_h u_\alpha$. Since linear independence is an open condition, $\{P(\tilde g) u_\alpha\}_{\alpha=1}^m$ forms a basis of $E_\gamma$ for all $\tilde g$ in some neighborhood $\mathcal U$ of $g$. Calculating $D_{g} F(g,A)$ from the expression $D \tilde u_\alpha (g)[h] = L \Delta'_h u_\alpha$ is straightforward. If $\tilde g \in \mathcal S_{m,k}$, then $\dim \ker F(\tilde g, \cdot) \geq k$, and hence $V \cap \ker F(\tilde g, \cdot) \neq \{0\}$. Thus, there exists $\tilde A \in \mathcal A$ with $F_1(\tilde g, \tilde A) = 0$.
\end{proof}

\begin{proof}[Proof of \Cref{prop:generalizednontransverseMetric}]
Let $A \in \mathbb R^{\frac{m(m+1)}{2}}$ be any nonzero symmetric matrix such that $F(g,A) = 0$, with $F$ defined as in \Cref{lemma:nontransverse defining function}. We will first show that $D_g F(g,A)$ is a pseudodifferential operator with non-smooth Schwartz kernel, because $g$ is not contained in the exceptional set $\mathcal S_\infty(\lambda)$. Pseudodifferential operators of \emph{finite rank} are smoothing operators, and therefore have smooth Schwartz kernels.
\footnote{Indeed, suppose $A$ is a $\Psi$DO such that $A(C^\infty(M)) \subseteq C^\infty(M)$ is finite dimensional. As $C^\infty(M)$ is dense in $\mathcal D'(M)$, $A(C^\infty(M))$ must be dense in $A(\mathcal D'(M))$. Since $A(C^\infty(M))$ is finite dimensional, now $A(\mathcal D'(M)) = A(C^\infty(M)) \subseteq C^\infty(M)$, so $A$ is a smoothing operator.}
Thus, $D_g F(g,A)$ must have infinite rank. Therefore, $\ker D_g F(g,A)$ has infinite codimension, which we then show implies $\widetilde{\mathcal S}_{m,k}$ has infinite codimension as well.

For a fixed $u \in E_\lambda(g)$, $L \circ \Delta '_{[\cdot]} u: \Gamma(S^2(T^{\ast}M)) \rightarrow C^{\infty}(M)$ is a pseudodifferential operator of order at most $-1$. This follows from the fact that $h \mapsto \Delta'_h u$ is a differential operator of order $1$, and $L$ is a pseudodifferential operator of order $-2$, as is easily seen from the functional calculus of pseudodifferential operators (see, e.g, \cite[Chapter XII, Theorem 1.3]{Taylor81}). In a slight abuse of notation, let $L \in \mathcal D'(M \times M)$ also denote the Schwartz kernel of the operator $L$. It satisfies the identity $\Delta_x L(x,y) = \lambda L(x,y) - \delta(x,y) + P(x,y)$, where $\delta(x,y)$ and $P(x,y)$ are the Schwartz kernels of the identity and the spectral projection onto $E_\lambda$, respectively.

We now compute the Schwartz kernel of $L \circ \Delta '_{[\cdot]} u$ for a fixed function $u\in C^\infty(M)$. This kernel is a distributional section of the bundle $\pi_2^\ast(\mathcal S^2(T^\ast M)) \rightarrow M \times M$, i.e., the pull-back of the bundle of symmetric $(0,2)$-tensors on $M$ along the projection to the second factor of $M\times M$. Its computation is most easily expressed in local coordinates, which suffices since we will only be interested in its behavior near the diagonal. In the following, we assume $h$ to be compactly supported in a coordinate patch.
\begin{align*}
    &L \circ \Delta '_{h} u = \int_M L(x,y) \left( \tfrac{1}{2} g^{ij}\partial_{y^i}\left(g_{k\ell} h^{k\ell}\right) \partial_{y^j} u - |g|^{-\frac12} \partial_{y^k} \left( |g|^{\frac12} h^{k\ell} \partial_{y^\ell} u \right)\right)|g|^{\frac12} dy \\
    &= \int_M \left[ \partial_{y^k} L(x,y) \partial_{y^\ell} u - \tfrac12 \big(L(x,y) \Delta u + g(\nabla_y L(x,y),\nabla u)\big) g_{k\ell} \right]h^{k\ell} |g|^{\frac12} dy
\end{align*}

With this expression in hand, the Schwartz kernel of $D_g F(g,A)$ is readily computed as the distributional section $K$ of $\pi_2^\ast(\mathcal S^2(T^\ast M)) \rightarrow M \times M$ given by
\begin{align}
\label{eq:schwartzkernel}
\begin{split}
    K(x,y)_{k\ell} = \ & A^{\alpha\beta} \left( \partial_{y^k} L(x,y) \partial_{y^\ell} u_\alpha(y) + \partial_{y^\ell} L(x,y) \partial_{y^k} u_\alpha(y) \right) u_\beta(x) \\- \ & A^{\alpha\beta} \big(\,\lambda L(x,y) u_\alpha(y)u_\beta(x) + g(\nabla_y L(x,y),\nabla u_\alpha(y))u_\beta(x) \,\big) g_{k\ell} 
\end{split}
\end{align}
Smoothness of $K(x,y)$ would imply $\rho(x,y) :=A^{\alpha\beta} L(x,y) u_\alpha(y)u_\beta(x)$ was smooth as well. If $\dim M = 2$ this follows directly from taking traces (in $y$):
    \begin{align*}
        -g^{k\ell} K(x,y)_{k\ell} = n \lambda \, A^{\alpha\beta} L(x,y) u_\alpha(y)u_\beta(x).
    \end{align*}
In dimension $3$ and above, we first apply the Laplacian in $y$ to $\rho(x,y)$, which yields 
    \begin{align*}
        \Delta_y \rho(x,y) = A^{\alpha\beta}\big[ \ &2\lambda L(x,y) u_\alpha(y)u_\beta(x) + 2g(\nabla_y L(x,y),\nabla u_\alpha(y))u_\beta(x) \\
        &- \delta(x,y) u_\alpha(y)u_\beta(x) + P(x,y)u_\alpha(y)u_\beta(x) \ \big].
    \end{align*}
    The third term is zero because $A^{\alpha\beta}u_\alpha(y)u_\beta(x)$ vanishes at the diagonal. Taking the trace of $K$, we obtain
    \begin{align*}
        g^{k\ell} K(x,y)_{k\ell} = (2-n) g(\nabla_y L(x,y),\nabla u_\alpha(y))u_\beta(x) - n \lambda \rho(x,y).
    \end{align*}
    Inserting this into the expression above shows that
    \begin{align*}
        \left( \tfrac{4\lambda}{n-2} + \Delta_y \right) \rho(x,y) = A^{\alpha\beta}P(x,y)u_\alpha(y)u_\beta(x) - \tfrac{2}{n-2} g^{k\ell} K(x,y)_{k\ell}.
    \end{align*}
    Elliptic regularity implies $\rho(x,y)$ is smooth in $y$ for fixed $x$. In fact, $\rho \in C^\infty(M\times M)$, since $\tfrac{4\lambda}{n-2} + \Delta_y$ is invertible on $C^\infty(M)$, and $C^\infty(M \times M) = C^\infty(M,C^\infty(M))$.

    For fixed $x \in M$, the function $A^{\alpha\beta} u_\alpha(y)u_\beta(x)$ is an eigenfunction of $\Delta_g$. Because it does not vanish identically on $M\times M$, there exists $x \in M$ where $A^{\alpha\beta} u_\alpha(y)u_\beta(x)$ is nonzero. Since $M$ is connected, this means $A^{\alpha\beta} u_\alpha(y)u_\beta(x)$ does not vanish to infinite order at $x=y$. If $\rho(x,y)$ was smooth, $L(x,y) = \rho(x,y)(A^{\alpha\beta} u_\alpha(y)u_\beta(x))^{-1}$ is an element of $\mathcal Q^\infty(M)$, which would contradict the assumption $g \in \mathcal S_{m,k}(\lambda)$. Thus, $\rho(x,y)$ is not smooth, which shows $K(x,y)$ is non-smooth as well. We conclude that $D_g F(g,A)$ has infinite rank.

    Consider the restricted map $F_1: \mathcal U \times \mathcal A \to C^\infty(M)$ defined in \Cref{lemma:nontransverse defining function}(3). Up to sign, there exists a unique $A \in \mathcal A$ such that $F_1(g, A) = 0$. Since $D_g F_1(g,A)$ has infinite rank, for any $N \in \mathbb N$ we may choose a linear map $P_N: C^\infty(M) \to \mathbb R^N$ such that $P_N \circ D_g F_1(g_0,A_0)$ is onto, and $\ker P_N \cap \mathrm{Im} \, D_A F_1(g,A) = \{0\}$. By the implicit function theorem (as in Appendix A, \Cref{proposition: appendix: submersion}), the set $(P_N \circ F)^{-1}(\{0\})$ is locally a Fr\'{e}chet submanifold of $\mathcal U \times \mathcal A$, which has codimension $N$. Its image under the projection $\pi: \mathcal U \times \mathcal A \to \mathcal U$ is still a manifold if its tangent space intersects $\ker D\pi$ trivially. This is true by construction, since $B \in \ker D\pi \cap \ker P_N \circ D_A F_1(g,A)$ would solve $F_1(g,B) = 0$ and be orthogonal to $A$, which was excluded by restriction to $\mathcal A$. Thus, after possibly shrinking $\mathcal U$, $\widetilde{\mathcal S}_{m,k}$ is contained in a Fr\'{e}chet submanifold of codimension $N - \frac{m(m+1)}{2} + 1$. Since $N$ was arbitrary, this concludes the proof in the non-conformal setting.

    The proof in the conformal case is almost identical, because only conformal perturbations were used to show $D_g F(g,A)$ has infinite rank. The key point of the preceding paragraphs was that $K(x,y)_{k\ell}$ is not smooth (unless $g \in \mathcal S^\infty(\lambda)$) because $g^{k\ell} K(x,y)_{k\ell}$ is not a smooth function, and the latter is precisely the Schwartz kernel of $D_g F(g,A)$ when restricted to conformal multiples of $g$.
\end{proof}

Next, we show that $\mathcal S_\infty = \bigcup_{\lambda \in (-\infty,0)} \mathcal S_\infty(\lambda)$ has infinite codimension. In fact, if $n=2$ or $n\geq 3$ is odd, this set is empty. If $n\geq 4$ is even, it may well be empty, too, but our methods only allow us to show it is very small.

\begin{lemma}
    \label{lemma: smooth kernel}
    Let $L(x,y)$ denote the Schwartz kernel of the pseudodifferential operator $L$ defined by $L(\lambda - \Delta_g) = (\lambda - \Delta_g)L = \mathbb I - P_\lambda$. Near the diagonal, it takes the form $L(x,y) = d(x,y)^{2-n} L_1(x,y) + \mathcal P(\lambda, g) \log(d(x,y)) + L_2(x,y)$, where 
    \begin{enumerate}[label=(\roman*)]
        \item $L_1(x,y)$ is smooth and nonvanishing at the diagonal,
        \item $\mathcal P(\lambda, g)$ is a universal polynomial of $\lambda$ and a finite jet of $g$ at $x$, which does not vanish identically on any conformal class, and
        \item $L_2(x,y)$ is bounded.
    \end{enumerate}
    In particular, $\mathcal S_\infty = \varnothing$ if $n=2$ or $n\geq 3$ is odd. If $n\geq 4$ is even, $g \in \mathcal S_\infty$ means $\mathcal P(\lambda, g) \equiv 0$ for some $\lambda \in \sigma(\Delta_g)$. The set of $g$ for which this occurs has, at least, infinite \tcodim, even when fixing the conformal class.
\end{lemma}

\begin{proof}
Denote by $P_{[2\lambda,0]}$ the spectral projection onto all eigenvalues contained in the interval $[2\lambda,0] \subseteq (-\infty,0]$. Consider the operator $R_\lambda = (\lambda - \Delta)^{-1}(\mathbb I - P_{[2\lambda,0]})$, which differs from the operator $L$ we are interested in by a smoothing operator. We will obtain the asymptotics of its Schwartz kernel by considering the heat propagator $e^{t \Delta} (\mathbb I - P_{[2\lambda,0]})$. This operator has very good decay properties in the limit $t \to \infty$: $\lVert e^{t \Delta} (\mathbb I - P_{[2\lambda,0]}) \rVert \lesssim e^{-2t(|\lambda|-\varepsilon)}$ in all reasonable norms, including $L^\infty$-bounds on the kernel. One then easily obtains the operator identity
\begin{align*}
    R_\lambda = \int_0^\infty e^{t (\Delta - \lambda)} (\mathbb I - P_{[2\lambda,0]}) dt,
\end{align*}
as the integrand is bounded as $t\rightarrow 0$ and decays exponentially as $t \rightarrow \infty$.

For $T>0$, the operator $\int_T^\infty e^{t (\Delta - \lambda)} (\mathbb I - P_{[2\lambda,0]}) dt$ is a smoothing operator, as its kernel is smooth.
The operator $\int_0^T e^{t (\Delta - \lambda)} P_{[2\lambda,0]} dt$ is likewise a smoothing operator, as it maps into $E_{[2\lambda,0]}$, which is a finite dimensional space of smooth functions.
Thus, up to a smoothing operator, 
\begin{align*}
    R_\lambda \sim \int_0^T e^{t (\Delta - \lambda)} dt = \int_0^T \sum_{k=0}^\infty \frac{|\lambda|^k}{k!} t^k e^{t\Delta} dt.
\end{align*}
The heat kernel has an asymptotic expression of the form
\begin{align*}
    k(t,p,q) = (4\pi t)^{-\frac n2}\exp\left(-\frac{d(p,q)^2}{4t}\right) \sum_{j=0}^\infty a_j(p,q) t^j,
\end{align*}
where $a_j(p,q) \in C^\infty(M \times M)$, and furthermore $a_j(p,p)$ is a polynomial in finitely many derivatives of the metric $g$. One obtains this e.g., from specializing Theorem 7.15 in \cite[pg. 101]{Roe1988} to the operator $d + d^\ast$ on $\Lambda( T^\ast M )$.
We collect powers of $t$ in the kernel of $e^{t (\Delta - \lambda)}$:
\begin{align*}
    \sum_{k=0}^\infty \frac{|\lambda|^k}{k!} t^k k(t,p,q) &= (4\pi t)^{-\frac n2}\exp\left(-\frac{d(p,q)^2}{4t}\right) \sum_{k=0}^\infty \sum_{j=0}^\infty \frac{|\lambda|^k}{k!} a_j(p,q) t^{k+j} \\
    &= (4\pi t)^{-\frac n2}\exp\left(-\frac{d(p,q)^2}{4t}\right) \sum_{k=0}^\infty \left( \sum_{j=0}^k \frac{|\lambda|^j}{j!} a_{k-j}(p,q) \right) t^k
\end{align*}
For the kernel $R(p,q)$ of $ R_\lambda$, one gets the following asymptotic expression by plugging in the heat kernel asymptotics into $\int_0^T e^{t (\Delta - \lambda)} dt$, and performing the change of variables $s = \frac{d(p,q)^2}{4t}$.
\begin{align*}
    R(p,q) \sim \sum_{k=0}^\infty 4^{-k-1}\pi^{-\frac n2} \left( \sum_{j=0}^k \frac{|\lambda|^j}{j!} a_{k-j}(p,q) \right) d(p,q)^{-n+2+2k} \int_{\frac{d(p,q)^2}{4T}}^{\infty} s^{\frac n2 - k - 2} e^{-s} ds
\end{align*}
Since $\int_{\frac{d(p,q)^2}{4T}}^{\infty} s^{\frac n2 - k - 2} e^{-s} ds \sim \Gamma(\frac n2 - k - 1)$ up to order $n - 2k - 2$, the singularity of $R(p,q)$ near the diagonal is given by 
\begin{align*}
    R(p,q) \sim & \sum_{k=0}^{\lfloor \frac{n-1}{2} \rfloor} 4^{-k-1}\pi^{-\frac n2}\Gamma(\frac n2 - k - 1) \left( \sum_{j=0}^k \frac{|\lambda|^j}{j!} a_{k-j}(p,q) \right) d(p,q)^{-n+2+2k} \\
    & + \frac{1+(-1)^n}{2^{n+3}\pi^{\frac n2}} \left( \sum_{j=0}^{n/2} \frac{|\lambda|^j}{j!} a_{k-j}(p,q) \right) \log(d(p,q)),
\end{align*}
with the logarithmic term present if $n$ is even. The power singularities are thus of the form $d(p,q)^{2-n}L_1(p,q)$ for some function $L_1 \in C^\infty(M\times M)$. The logarithmic term vanishes only if $\mathcal P(\lambda, g)(p) := \sum_{j=0}^{n/2} \frac{|\lambda|^j}{j!} a_{k-j}(p,p)$ vanishes identically.

We now show that, unless $\mathcal P(\lambda,g) \equiv 0$, the kernel $L$ cannot be a quotient of smooth functions. Recall the ring $\mathcal Q^\infty(M)$ of quotients of smooth functions whose denominator does not vanish to infinite order. We wish to show there exists $x \in M$ such that $L(x,y) \not \in \mathcal Q^\infty(M)$. Consider an arbitrary fraction $\frac \varphi \psi \in Q^\infty(M)$. Along a curve $\gamma:(-\varepsilon,\varepsilon) \to M$ where $\psi \circ \gamma$ does not vanish to infinite order, $\left( \frac \varphi\psi \circ \gamma \right) (t) = t^k h(t)$ for some $k \in \mathbb Z$ and $h \in C^\infty(-\varepsilon,\varepsilon)$ with $h(0) \neq 0$.

Thus, if $n = 2$ or $n \geq 3$ is odd, its leading order singularity already precludes $L(x,y)$ from lying in $\mathcal Q^\infty(M)$. If $n \geq 4$ is even, $d(x,y)^{2-n} L_1(x,y) \in \mathcal Q^\infty(M)$, and hence $L(x,y) \in \mathcal Q^\infty(M)$ if and only if $\mathcal P(\lambda, g) \log(d(x,y)) + L_2(x,y) \in Q^\infty(M)$. Thus, $L(x,y) \in \mathcal Q^\infty(M)$ implies that $\mathcal P(\lambda, g)$ vanishes identically on $M$. Since $\mathcal P(\lambda, g)$ is a purely local quantity, which is generically nonzero, the set of metrics where $\mathcal P(\lambda, g)$ vanishes identically has infinite transversality codimension, and thus also $\mathcal S^\infty$.

The only caveat is that $\mathcal P(\lambda, g)$ might vanish \emph{as a polynomial}, which we can exclude by exhibiting a single metric where $\mathcal P(\lambda, g) \neq 0$. The euclidean metric already suffices: The fundamental solution to Helmholtz' equation in Euclidean space, $(1 + \Delta_{\mathrm{Euc}}) u = \delta_0$, is given by $u(x) = c_n |x|^{1-\frac{n}{2}} Y_{\frac{n}{2}-1}(|x|)$, where $Y_{\frac{n}{2}-1}$ is a Bessel function of the second kind and $c_n$ a nonzero dimensional constant. If $n$ is even, the full asymptotic expansion of $Y_{m}(z)$ for small arguments given in \cite[Eq. 9.1.11]{AbramowitzStegun} implies
\begin{align*}
    Y_m(z) = -\frac{\left( \frac{1}{2}z \right)^{-m} }{ \pi} \sum_{k=0}^{m-1} \frac{(m-k-1)!}{k!}(\tfrac14 z^2)^k + \frac2\pi \log(\tfrac 12 z) J_m(z) + \mathcal O(1),
\end{align*}
where $J_m$ is a Bessel function of the first kind. Since $J_{\frac n2-1}$ vanishes to order $\frac n2-1$, and $u(x) = c_n |x|^{1-\frac{n}{2}} Y_{\frac{n}{2}-1}(|x|)$, we conclude that $\mathcal P(\lambda, g_{\mathrm{Euc}})$ is a nonzero constant.

For the last claim, fix a conformal class $\mathcal G(M,g_0)$. We must show that $\mathcal P(\lambda, g)$ does not vanish identically on $\mathcal G(M,g_0)$. Take an arbitrary point $p \in M$ and consider $g_\varepsilon = \varphi_\varepsilon g_0$, where $\varphi_\varepsilon \equiv \varepsilon^{-2}$ on $\mathbb B_{\varepsilon}^{g_0}(p)$ and $\varphi_\varepsilon \equiv 1$ outside $\mathbb B_{2\varepsilon}^{g_0}(p)$. In normal coordinates at $p$, $g_\varepsilon$ and all its derivatives converge to the Euclidean metric as $\varepsilon$ approaches $0$. Since $\mathcal P(\lambda, g_{\mathrm{Euc}}) \neq 0$, there must be some $\varepsilon$ such that $\mathcal P(\lambda, g_{\varepsilon})$ does not vanish. Hence, $\mathcal P(\lambda, g)$ does not vanish identically on any conformal class.
\end{proof}

The full non-crossing rule (\Cref{theorem:nonCrossing}) follows by ``stitching together'' the local codimension bounds from \Cref{prop:generalizednontransverseMetric} and \Cref{lemma: smooth kernel} (which, along the way, yields \Cref{theorem:nontransverse}) with the local structure theorem near nondegerate eigenvalues, \Cref{theorem:transverse structure theorem}.

\begin{proof}[Proof of \Cref{theorem:nontransverse} and \Cref{theorem:nonCrossing}]
    By \Cref{lemma: smooth kernel}, the exceptional set $\mathcal S_\infty$ has infinite \tcodim. It remains to show the same for $\mathcal S_{m,k} = \bigcup_{\lambda \in (-\infty,0)} S_{m,k}(\lambda)$.
    
    Fix $N \in \mathbb N$. We now show that $\mathcal S_{m,k}$  has \tcodim \ $N$, and thus, as $N$ was arbitrary, of infinite \tcodim. Let us order $(m,k)$ lexicographically, i.e., $(m',k') \succeq (m,k)$ if $m' > m$ or $m'=m$, $k' \geq k$.
    For $R \in \mathbb N$, let $\overline{\mathcal S}_{m,k;R}$ denote the set of metrics $g$ such that $g \in \mathcal S_{m,k}(\lambda)$ for some $\lambda \in [-R,0)$, but there exist no $(m',k') \succeq (m,k)$ such that $g \in \mathcal S_{m',k'}(\lambda')$ for some $\lambda' \in [-R,0)$. Since both $m$ and $k$ are dimensions of the null space of some linear system varying continuously with $g$, and hence upper semicontinuous, this definition achieves the following: For any $g \in \overline{\mathcal S}_{m,k;R}$, there exists an open neighborhood $\mathcal U_0$ of $g$ with $\mathcal U_0 \cap \overline{\mathcal S}_{m',k';R} = \varnothing$ for any $(m',k') \succeq (m,k)$.
    
    Consider $g \in \overline{\mathcal S}_{m,k;R}$ and $\mathcal U_0$ as above. Let $\lambda_1,\ldots,\lambda_\ell \in [-R,0)$ denote all eigenvalues for which $g \in \mathcal S_{m,k}(\lambda)$. \Cref{prop:generalizednontransverseMetric} yields neighborhoods $\mathcal U_j$ of $g$, $j = 1,\ldots,\ell$, and corresponding submanifolds $\widetilde{\mathcal S}_j \subseteq \mathcal U_j$ of codimension $N$ such that $\left(\bigcap_{j=0}^k \mathcal U_j\right) \cap \overline{\mathcal S}_{m,k;R} \subseteq \bigcup_{j=1}^k \widetilde{\mathcal S}_j$. Thus, by \Cref{lemma: appendix: local transversality}, $\overline{\mathcal S}_{m,k;R}$ has \tcodim \ $N$. The set $\mathcal S_{m,k}$ is contained in a countable union of such sets:
    \begin{align*}
        \mathcal S_{m,k} \subseteq \hspace{-5mm} \bigcup\limits_{\substack{(m',k') \succeq (m,k), \\ R \in \mathbb N}}  \hspace{-4mm}  \overline{\mathcal S}_{m',k';R}.
    \end{align*}
    As $N$ was arbitrary, this shows $\mathcal S_{m,k}$ has infinite \tcodim, and so is the set $\mathcal S = \mathcal S_\infty \cup \left( \bigcup_{m,k} \mathcal S_{m,k} \right)$ of all metrics admitting an unstable eigenvalue. Thus, \Cref{theorem:nontransverse} is proven.

    Stable eigenvalues are dealt with in much the same way. Introduce $\overline{\mathcal T}_{m;R}$, the set of metrics admitting a stable eigenvalue $\lambda \in [-R,0)$ of multiplicity exactly $m$, but no $\lambda \in [-R,0) \cap \sigma(\Delta_g)$ which is unstable, or stable but of higher multiplicity than $m$. Analogous to before, there exist a neighborhood $\mathcal U$ around any $g \in \overline{\mathcal T}_{m;R}$, and submanifolds $\mathcal T_1, \ldots, \mathcal T_\ell \subseteq \mathcal U$ of codimension $\frac{m(m+1)}{2}-1$, such that $\mathcal U \cap \overline{\mathcal T}_{m;R} = \bigcup_{j=1}^\ell \mathcal T_j$. \Cref{lemma: appendix: local transversality} implies $\overline{\mathcal T}_{m;R}$ has \tcodim \ $\frac{m(m+1)}{2}-1$. The set of metrics admitting an eigenvalue of multiplicity at least $m$ is covered by the union of $\mathcal S$ and the sets $\overline{\mathcal T}_{m';R}$, $m' \geq m, R \in \mathbb N$. It therefore has \tcodim \ $\frac{m(m+1)}{2}-1$.
\end{proof}

The maximal non-crossing rule in $C^\infty$ regularity implies a maximal non-crossing rule in finite regularity. Recall Courant's min-max characterization of the $k^{th}$ eigenvalue \cite[Theorem 4.5.12]{LableeSpectralTheory}:

\begin{align}
    - \lambda_k = \min_{\substack{V \subseteq H^1(M) \\ \dim V = k}} \, \max_{u \in V \setminus\{0\} }\frac{\int_M g(du,du) d\mu_g}{\int_M u^2 d\mu_g}
\end{align}

This definition is meaningful for continuous metrics, and the eigenvalues thus obtained agree with those of the operator $\Delta_g: H^2(M) \to L^2(M)$ if $g$ is continuously differentiable. For $\ell \geq 0$, let $\mathcal{G}^\ell (M)$ denote the Banach manifold of $C^\ell$ metrics on $M$. The eigenvalue functional $\lambda_k$ is continuous on $\mathcal{G}^\ell (M)$ for all $\ell \geq 0$. This follows from Lemma 4.5.13 in \cite{LableeSpectralTheory}, which is stated in smooth regularity, but works just as well in $\mathcal G^0(M)$.

\begin{corollary}
\label{corollary: finite regularity}
    Let $\ell \geq 0$, and consider the Banach manifold $\mathcal{G}^\ell (M)$  of $C^\ell$ metrics on $M$. The set of metrics which have at least one eigenvalue of multiplicity at least $m$ has \tcodim \ $\frac{m(m+1)}{2}-1$.
\end{corollary}

\begin{proof}
    Denote the set of metrics admitting an eigenvalue $\lambda\in [-R,0)\cap \sigma(\Delta_g)$ of multiplicity at least $m$ by $\mathcal{D}_{m;R}\subset \mathcal{G}^\ell (M)$. This set is closed, since eigenvalues move continuously under deformations of the metric. Consequently, the set of smooth maps from a compact $k$-dimensional manifold $N$ to $\mathcal{G}^\ell (M)$ that avoids $\mathcal{D}_{m;R}$ is open.
    
    We now show it is dense if $k < \frac{m(m+1)}{2}-1$. For this, let $S_\varepsilon : \mathcal{G}^\ell (M)\to \mathcal{G}^\ell (M)$ be a family of smoothing operators converging to the identity in the strong operator topology. Since $N$ is compact, we can ensure that $\lVert f- S_\varepsilon f \rVert_{C^\ell}$ is arbitrarily small. Now, $S_\varepsilon f$ can be perturbed in $\mathcal{G}^\infty (M)$ to avoid $\mathcal{D}_{m;R}$ by \Cref{theorem:nonCrossing}.
    
    The set of smooth maps from $N$ to $\mathcal{G}^\ell (M)$ which avoid eigenvalues of multiplicity at least $m$ is thus a countable intersection of dense open sets, and hence Baire generic. For non-compact $N$, the result follows by a countable compact exhaustion.
\end{proof}

\section{Low multiplicities and examples}

This section is concerned with natural questions and examples associated to the phenomenon of (conformally) stable eigenvalues. The simplest examples are flat tori and the round sphere. Thus, already in 1988, Colin de Verdi\`ere classified the unstable and conformally unstable eigenvalues on the square torus and the two-dimensional round sphere while investigating the strong Arnold hypothesis \cite{CdV1988}. We discuss these examples, in two \emph{and} higher dimensions, in sections \ref{sec:torus} and \ref{sec:sphere}, respectively. Section \ref{section:low multiplicities and examples} is concerned with eigenvalues of low multiplicity, and Section \ref{sec:misc} lists further examples and comparisons.

\subsection{Nonexistence of unstable eigenvalues of low multiplicity}
\label{section:low multiplicities and examples}

Unstable eigenvalues must have fairly high multiplicity. This observation goes back at least to \cite[Theorem 4.3]{Micheletti1998}, which essentially says that there are no conformally unstable eigenvalues of multiplicity 2 in two dimensions. It is natural to ask for the lowest possible multiplicity of a (conformally) unstable eigenvalue. \Cref{theorem: nonexistence of low multiplicity nontransverse eigenvalues} gives an almost complete answer.

In Section \ref{sec:torus}, we will see examples of conformally unstable eigenvalues of multiplicity $4$ and unstable eigenvalues of multiplicity $8$. It is an interesting question whether unstable eigenvalues of multiplicity $7$ exist (see Question \ref{q:multiplicity7}).

\begin{proof}[Proof of \Cref{theorem: nonexistence of low multiplicity nontransverse eigenvalues}]
The first claim will follow from the fact that, on a connected Riemannian manifold, the product of two linearly independent real valued Laplace eigenfunctions cannot be nonnegative. To see this, let $u,v \in C^\infty(M)$ be two nonzero eigenfunctions such that $u(x)v(x) \geq 0$ for all $x \in M$. Near any $x \in u^{-1}(\{0\})$, $u$ must take both positive and negative values by the maximum principle, which implies that $v(x) = 0$ as well. Thus, $u$ and $v$ share the same zero set. Let $\Omega$ be a connected component of $u^{-1}(\mathbb R \setminus \{0\})$. The functions $u \rvert_\Omega$ and $v \rvert_\Omega$ are positive solutions to the Dirichlet eigenvalue problem on $\overline \Omega$. Any such solution corresponds to the lowest eigenvalue, which is always simple, so $u = t v$ for some $t > 0$. 

Consider a conformally unstable eigenvalue $\lambda \in \sigma(\Delta_g)$, and associated linearly independent eigenfunctions $u_1,\ldots,u_k \in E_\lambda$ such that $\sum_{j=1}^k \epsilon_j u_j^2 = 0$  for some $\epsilon_1,\ldots,\epsilon_k \in \{-1,1\}$. Not all $\epsilon_j$ can be $1$, so $k \geq 2$. If $k=2$, $u_1^2 = u_2^2$, which by unique continuation implies $u_1 = \pm u_2$, a contradiction. Suppose now that $k=3$. Without loss of generality, $u_1^2 = u_2^2 + u_3^2$. Then $(u_1-u_2)(u_1+u_2) = u_3^2 \geq 0$, which is also impossible, as shown in the first paragraph. A conformally unstable eigenvalue must therefore be of multiplicity at least $4$.

If $\lambda \in \sigma(\Delta_g)$ is unstable, there exist linearly independent eigenfunctions $u_1,\ldots,u_k$ such that $\sum_{j=1}^k \epsilon_j u_j^2 = 0$ and $\sum_{j=1}^k \epsilon_j \, du_j \! \otimes \! du_j = 0$. Without loss of generality, $\epsilon_1 = \ldots = \epsilon_\ell = -1$, and $\epsilon_{\ell+1} = \ldots = \epsilon_k = 1$. The same argument as before shows $\min(\ell,k-\ell) \geq 2$.  Consider the maps $\Phi:M \rightarrow \mathbb R^\ell$, $\Phi = (u_1,\ldots,u_\ell)$ and $\Psi:M \rightarrow \mathbb R^{k-\ell}$, $\Psi = (u_{\ell+1},\ldots,u_k)$. Then $\lvert \Phi(x) \rvert = \lvert \Psi(x) \rvert$ for all $x \in M$, as well as $\Phi^\ast g_{\mathbb R^\ell} = \Psi^\ast g_{\mathbb R^{k-\ell}}$, with $g_{\mathbb R^\ell}$ and $g_{\mathbb R^{k-\ell}}$ denoting the standard metrics on $\mathbb R^\ell$ and $\mathbb R^{k-\ell}$, respectively. Both conditions together imply that the normalized maps $\tilde \Phi: M \to \mathbb S^{\ell-1}$, $\tilde \Phi(x) = |\Phi(x)|^{-1}\Phi(x)$ and $\tilde \Psi: M \to \mathbb S^{k-\ell-1}$, $\tilde \Psi(x) = |\Psi(x)|^{-1}\Psi(x)$ satisfy a similar pull-back metric condition: $\tilde \Phi^\ast g_{\mathbb S^{\ell-1}} = \tilde \Psi^\ast g_{\mathbb S^{k-\ell-1}}$.

To show the last claim, consider a vector field $V \in \Gamma(TM)$. Our assumptions on $\Phi$ and $\Psi$ imply $|V\Phi|^2 = |V\Psi|^2$, and furthermore
\begin{align*}
    \Phi \cdot (V\Phi) = V\left(\tfrac12 \lvert \Phi\rvert^2 \right) = V\left(\tfrac12 \lvert \Psi\rvert^2 \right) = \Psi \cdot (V\Psi).
\end{align*}
These identities imply the desired claim, via polarization, as
\begin{align*}
    \left(\tilde \Phi^\ast g_{\mathbb S^{\ell-1}}\right)(V,V) &= \left\lvert \Phi \right\rvert^{-2} \left\lvert V\Phi - \left( \left\lvert \Phi \right\rvert^{-2} \Phi \cdot (V\Phi) \right) \Phi \right\rvert^2 \\
    &= \left\lvert \Phi \right\rvert^{-2} \left\lvert V \Phi \right\rvert^{2}
    - \left\lvert \Phi \right\rvert^{-4} \left( \Phi \cdot (V \Phi)\right)^2 = \left( \tilde \Psi^\ast g_{\mathbb S^{k-\ell-1}}\right)(V,V).
\end{align*}

Of course, it may happen that $\Phi(x) = \Psi(x) = 0$ for some $x \in M$, and hence, $\tilde \Phi$ and $\tilde \Psi$ are not defined. However, they are defined on a dense open subset. Since all following arguments are local, this is sufficient.

Suppose there exists an open set $U$ where $\mathrm{rk} \, D\tilde \Phi = 1$. Possibly after shrinking $U$, the rank theorem implies there exists $\tau \in C^\infty(U)$ and a function $F: \mathbb R \to \mathbb S^{\ell-1}$ such that $\tilde \Phi = F \circ \tau$. For any $V \in \ker D\tau$, we have $D\tilde \Phi(V) = 0$, and thus also $D\tilde \Psi(V) = 0$, because $\tilde \Phi^\ast g_{\mathbb S^{\ell-1}} = \tilde \Psi^\ast g_{\mathbb S^{k-\ell-1}}$. Hence, there also exists $G: \mathbb R \to \mathbb S^{k-\ell-1}$ with $\tilde \Psi = G \circ \tau$. Denote $\rho = |\Phi| = |\Psi|$. Then the preceding calculation shows there exist functions $g_1,\ldots,g_k$ such that $u_j = (g_j \circ \tau) \rho$ for $j = 1,\ldots,k$. The eigenfunction equation becomes
\begin{align*}
    0 &= (\lambda - \Delta) \big((g_j \circ \tau) \rho\big) \\ &=
    (g_j \circ \tau) ( \lambda \rho - \Delta \rho) - (g_j''\circ \tau) |\nabla \tau|^2 \rho - (g_j' \circ \tau) (\rho \Delta \tau + \nabla \rho \cdot \nabla \tau).
\end{align*}
Note that $|\nabla \tau|^2 \rho$ does not vanish, by assumption. There can be at most $2$ linearly independent functions $g_j$ which satisfy this equation. Indeed, consider a curve $\gamma: (-\varepsilon,\varepsilon) \to M$ such that $\tau \circ \gamma = \mathrm{id}$. Then $g_j$ satisfies the linear ordinary differential equation $g_j''(t) = A(t) g_j(t) - B(t) g_j'(t)$, where
\begin{align*}
    A &= \left( \frac{\lambda \rho - \Delta \rho}{|\nabla \tau|^2 \rho} \right) \circ \gamma, \ B = \left(\frac{\rho \Delta \tau + \nabla \rho \cdot \nabla \tau}{|\nabla \tau|^2 \rho} \right)\circ \gamma.
\end{align*}
If $k \geq 3$, this would imply a linear dependence between $u_1,\ldots,u_k$ along $\gamma$, which is determined solely by the values of $u_1,\ldots,u_k$ and their first derivatives along $\gamma$ at $0$. Thus, it extends to all curves which agree up to first order with $\gamma$ at $0$, and hence, to an open subset of $U$. By unique continuation, the linear dependence now extends to the whole of $M$. This contradicts our assumption that $u_1,\ldots,u_k$ are linearly independent. The case $k \leq 2$ was already excluded earlier, so there exists no open set $U$ where $\mathrm{rk} \, D\tilde \Phi = 1$.

We deduce that $\ell \geq 3$ and $k \geq \ell+3$, and that there exists an open set $U \subseteq M$ where $D\tilde \Phi$ and $D\tilde \Psi$ are of rank at least $2$. Consider the case $\ell = k-\ell = 3$. On $U$, $\tilde \Phi$ and $\tilde \Psi$ are submersions with $\tilde \Phi^\ast g_{\mathbb S^2} = \tilde \Psi^\ast g_{\mathbb S^{2}}$, so there exists a local isometry $F$ of $\mathbb S^2$ with $ \tilde \Phi = F \circ \tilde \Psi$. Since all local isometries of $\mathbb S^2$ are linear, there exists $A \in O(3)$ such that $\tilde \Phi = A \tilde \Psi$ on $U$. Recalling that $\lvert \Phi \rvert = \lvert \Psi \rvert$, we find that $\Phi = A \Psi$ as well, which, by the unique continuation principle, is a contradiction, as $u_1,\ldots,u_k$ are linearly independent.

In conclusion, we have shown that if $u_1,\ldots,u_\ell,\ldots,u_k \in E_\lambda$ are linearly independent eigenfunctions with $\sum_{j=1}^\ell u_j^2 = \sum_{j=\ell+1}^k u_j^2$, then $\min(\ell,k-\ell) \geq 2$, and if additionally $\sum_{j=1}^\ell d u_j \otimes du_j = \sum_{j=\ell+1}^k d u_j \otimes du_j$, then $\min(\ell,k-\ell) \geq 3$ and $\max(\ell,k-\ell) \geq 4$.
\end{proof}

\begin{remark}
    In the proof of \Cref{theorem: nonexistence of low multiplicity nontransverse eigenvalues}, the exact multiplicity of the eigenvalue was never used. This suggests to call an eigenvalue \emph{conformally} $(\ell,m)$-\emph{unstable}, if there exist corresponding eigenfunctions $u_1,\ldots,u_\ell,u_{\ell+1},\ldots,u_{\ell+m}$ which are linearly independent and satisfy
    $\sum_{j=1}^\ell u_j^2 = \sum_{j=\ell+1}^{\ell+m} u_j^2$, and $(\ell,m)$-\emph{unstable} if, additionally, $\sum_{j=1}^\ell du_j \otimes du_j = \sum_{j=\ell+1}^{\ell+m} du_j \otimes du_j$. Then there are no conformally $(0,m)$- or $(1,m)$-unstable eigenvalues, and no $(2,m)$- or $(3,3)$-unstable eigenvalues.

    This allows for an alternative proof of \Cref{theorem:nonCrossing} for multiplicities up to $6$, which works in a $C^k$ setting just as well as in the present $C^\infty$ setting. Smoothness is only needed for the proof that the set of unstable eigenvalues has infinite codimension (Section \ref{section: nontransverse eigenvalues}), which relies on microlocal analysis.
\end{remark}

\subsection{Examples with many unstable eigenvalues: Flat tori}
\label{sec:torus}

A lattice of full rank $\Lambda \subseteq \mathbb R^n$ gives rise to a torus, $\mathbb T^n = \mathbb R^n/\Lambda$, which inherits its metric from $\mathbb R^n$.
A full basis of eigenfunctions of its Laplace operator is given by $u_\kappa(x) = e^{2\pi i \kappa \cdot x}$, where $\kappa$
ranges over all elements of $\Lambda^\ast = \{\kappa \in \mathbb R^n: \kappa \cdot v \in \mathbb Z \text{ for all } v \in \Lambda\}$, the dual lattice.
The corresponding eigenvalues are $-|\kappa|^2$, $\kappa \in \Lambda^\ast$. Thus, all eigenvalues have even multiplicity. There are many lattices that give rise to tori with eigenvalues of arbitrarily high multiplicity,
e.g., the square lattice $\mathbb Z \oplus \mathbb Z$, or the triangular lattice $\mathbb Z \oplus e^{\frac{2i\pi}{3}} \mathbb Z$. There are examples of flat tori where some eigenvalues have multiplicity $2$, but there are also eigenvalues of higher multiplicity, e.g., if $\Lambda = \mathbb Z \oplus 2 \mathbb Z$. In this case, there will be both conformally unstable and stable eigenvalues (as a consequence of the following proposition).

\begin{proposition}
\label{ex:nontransverse}
    Consider the torus $\mathbb T^n = \mathbb R^n/\Lambda$, where $\Lambda$ is a lattice of full rank. An eigenvalue of its Laplace-Beltrami operator on the corresponding torus is unstable if (and in dimension 2 only if) its multiplicity is at least $n(n+1)+2$. It is conformally unstable if and only if its multiplicity is at least $4$.
\end{proposition}

\begin{proof}
    Let $\lambda$ be a Laplace-Beltrami eigenvalue of $\mathbb T^n$. If it has multiplicity $2$, it cannot be conformally unstable, by \Cref{theorem: nonexistence of low multiplicity nontransverse eigenvalues}. If its multiplicity is at least $4$, there exist linearly independent elements $\kappa_1, \kappa_2 \in \Lambda^\ast$ with $|\kappa_1|^2 = |\kappa_1|^2$, which correspond to eigenfunctions $\cos(\kappa_j \cdot x)$ and $\sin(\kappa_j \cdot x)$, $j = 1,2$. Since $\cos(\kappa_1 \cdot x)^2 + \sin(\kappa_1 \cdot x)^2 = \cos(\kappa_2 \cdot x)^2 + \sin(\kappa_2 \cdot x)^2 = 1$, $\lambda$ is conformally unstable.

    To show the first part of the statement, let $\lambda$ be an eigenvalue of multiplicity $2m$. The corresponding eigenfunctions are 
    $\cos(\kappa_j \cdot x), \sin(\kappa_j \cdot x)$, where $\kappa_j \in \Lambda^\ast$ are lattice elements with $-|\kappa_j|^2 = \lambda$, and $j = 1,\ldots,m$.
    We will look for a nontrivial choice of coefficients $\mu_1,\ldots,\mu_m$ such that
    \begin{align*}
        &\sum_{j=1}^m \mu_j \left( \cos(\kappa_j \cdot x)^2 + \sin(\kappa_j \cdot x)^2 \right) = 0 \text{ and }\\
        &\sum_{j=1}^m \mu_j \left( \left(d\cos(\kappa_j \cdot x)\right)^2 + \left(d\sin(\kappa_j \cdot x)\right)^2 \right) = 0,
    \end{align*}
    which, by \Cref{proposition:submersion equivalent linear nondeg}, would imply that $\lambda$ is unstable. The first line is equivalent to $\sum_{j=1}^m \mu_j = 0$, while the second line reduces to $\sum_{j=1}^m \mu_j \kappa_j \kappa_j^T = 0$, a system of $\frac{n(n+1)}{2}$ equations. In view of Proposition \ref{proposition: further reduction} it is not surprising that the first equation is linearly dependent on the remaining ones: $\mathrm{tr}(\kappa_j \kappa_j^T) = |\kappa_j|^2 = -\lambda$. Thus, if $m \geq \frac{n(n+1)}{2} + 1$, there exists a nontrivial solution to this system, and $\lambda$ is unstable.

    It may be that not all components of the system $\sum_{j=1}^m \mu_j \kappa_j \kappa_j^T = 0$ are linearly independent, in which case $\lambda$ is unstable even if its multiplicity is less than $n(n+1)+2$. However, in dimension two, an eigenvalue of multiplicity less than $8$ has multiplicity $6$, and therefore, due to \Cref{theorem: nonexistence of low multiplicity nontransverse eigenvalues}, is not unstable.
\end{proof}

\begin{proposition}
\label{ex:nontransverse_structure}
    Consider the first positive eigenvalue $\lambda$ of the Laplace-Beltrami operator on the square torus $\mathbb T^2 = \mathbb R^2/\mathbb Z^2$ with its standard metric $g_0$. Then the set of metrics $g$ in the conformal class $\mathcal G(\mathbb T^2, g_0)$ such that $\lambda$ retains multiplicity $4$ has codimension $9$ near $g_0$, but is not a submanifold of $\mathcal G(\mathbb T^2, g_0)$.
\end{proposition}

This observation rests on the surprisingly subtle fact that the eigenvalue $\lambda$ satisfies a weaker form of transversality, also introduced by Colin de Verdi\`ere \cite{CdV1988}. Called the \emph{Weak Arnold Hypothesis} in \cite{CdV1988}, it is also known as \emph{weak stability}, see, e.g., \cite{Anne1994}.

\begin{definition}
    Let $(M,g_0)$ be a smooth, closed Riemannian manifold. We say an eigenvalue $\lambda \in \sigma(\Delta_{g_0})$ of multiplicity $m$ is \emph{weakly stable} if, given $\mathcal U \subseteq \mathcal G(M)$ and $\pi: \mathcal U \to \mathbb R^{\frac{m(m+1)}{2}}$ as defined in \Cref{section: local structure theorem}, there exists $\varepsilon > 0$ such that for every continuous map $\tilde \pi: \mathcal U \to \mathbb R^{\frac{m(m+1)}{2}}$ with $\|\tilde \pi - \pi\|_{C^0(\mathcal U)} < \varepsilon$, we have $\tilde \pi^{-1}(\{\lambda \mathrm{id}\}) \neq \varnothing$.
    
    If the analogous statement holds with $\mathcal G(M)$ replaced by $\mathcal G(M,g_0)$, we call $\lambda$ \emph{weakly conformally stable}.
\end{definition}

\begin{lemma}
\label{ex:WAH_2d}
    Let $(M,g_0)$ be a two-dimensional Riemannian manifold. Let $\lambda$ be a nonzero eigenvalue of $\Delta_{g_0}$ of multiplicity $m$. Assume $\lambda$ is conformally unstable, and that $\{A \in \mathbb R^{\frac{m(m+1)}{2}}: A^{\alpha\beta} u_{\alpha}u_\beta = 0\}$ is one-dimensional. Assume furthermore that $M$ possesses an isometric involution $\Phi$ such that 
    \begin{align*}
        A^{\alpha\beta} u_\alpha(x)u_\beta(y) = - A^{\alpha\beta} \Phi^\ast u_\alpha(x) \Phi^\ast u_\beta(y)
    \end{align*}
    Then $(g_0,\lambda)$ is weakly conformally stable, and the \tcodim\ of the set of metrics $g \in \mathcal G(M,g_0)$ such that $\lambda$ retains multiplicity $m$ is exactly $\frac{m(m+1)}{2}-1$.
\end{lemma}

\begin{proof}
Consider $\mathcal U \subseteq \mathcal G(M, g_0)$ and $\pi: \mathcal U \to \mathbb R^{\frac{m(m+1)}{2}}$ as in \Cref{section: local structure theorem}. The set $\mathcal S_m \subseteq \mathcal U$ of metrics for which $\lambda$ retains multiplicity $m$ is $\pi^{-1}(\{\mu \cdot \mathrm{id}_m, \mu \in \mathbb R\})$.

Let $A$ be a nonzero matrix -- unique up to scaling -- which satisfies $A^{\alpha\beta} u_{\alpha}u_\beta = 0$. As discussed in \Cref{section: local structure theorem},
\begin{align*}
    D\pi(g_0)[T_{g_0}\mathcal U] = \{B \in \mathbb R^{\frac{m(m+1)}{2}}: A^{\alpha\beta} B_{\alpha\beta} = 0\}.
\end{align*}
Denote $K = \ker D \pi(g_0)$, so that $D\pi(g_0)[T_{g_0}\mathcal U] = A^\bot$. An essentially finite-dimensional argument now shows that $\mathcal S_m$ has finite codimension in the sense of \Cref{def:tcodim} provided the quadratic form $A^{\alpha\beta} (D^2\pi(g_0)[h,h])_{\alpha\beta}$ is indefinite on $K$.

Indeed, let $\phi: [-1,1]^{N}\times[-1,1]^2 \to \mathcal U$ be any map such that $\phi(0,0,0) = g_0$, $D_1\phi(0,0,0)[T_0 \mathbb R^{N} ,0,0] = A^\bot$, $D_2\phi(0,0,0) = D_3\phi(0,0,0) = 0$ and 
\begin{align*}
    A^{\alpha\beta} D_2^2 \phi(0,0,0)_{\alpha\beta} < 0 < A^{\alpha\beta} D_3^2\phi(0,0,0)_{\alpha\beta}
\end{align*}
By the implicit function theorem, there exists a submanifold $S \subseteq [-1,1]^{N}\times[-1,1]^2$ such that $P_{A^\bot} \phi = \lambda \, \mathrm{id}$ on $S$ and furthermore $\{0\} \times \mathbb R^2 \subseteq TS$. On $S$, the indefiniteness of $A^{\alpha\beta} D^2 \phi(0,0,0)_{\alpha\beta}$ ensures that $A^{\alpha\beta} \phi_{\alpha\beta}$ takes both positive and negative values, so by the intermediate value theorem $P_{A} \phi = 0$ somewhere on $S$. The existence of $S$ and the fact that $A^{\alpha\beta} \phi_{\alpha\beta}$ takes both positive and negative values is stable under $C^1$-perturbation of $\phi$, so $\mathcal S_m$ has \tcodim at most $N+1$.

Recall that 
\begin{align*}
    \pi(g) &=
    Q(g)^{-1} S(g) P(g_0) \Delta(g) P(g) Q(g),
\end{align*}
where $P(g)$ is the Riesz projection corresponding to a circle $\Gamma \subseteq \mathbb C$ encircling $\lambda$, $Q: \mathbb R^{m} \to E_\lambda(g_0)$ is a map such that $\left(P(g) Q(g) e_\alpha, P(g) Q(g) e_\beta \right)_{L^2(M,g)} = \delta_{\alpha\beta}$ and $S: E_\lambda(g_0) \to E_\lambda(g_0)$ is implicitly defined by $S(g) P(g_0) P(g) P(g_0) = P(g_0)$. Differentiating twice and observing that, since $\mathrm{id}_m = Q(g)^{-1} S(g) P(g_0) P(g) Q(g)$ and $P(g_0) \Delta(g_0) = \lambda P(g_0)$, all terms where no derivative falls on $\Delta$ cancel, we find
\begin{align*}
    \ddot \pi &= Q^{-1} P \left( \ddot \Delta + 2(\dot \Delta \dot P + \dot \Delta \dot Q Q^{-1} - \dot Q Q^{-1} P \dot \Delta + \dot S P \dot \Delta) \right) Q \\
    &= Q^{-1} P \left( \ddot \Delta + 2(\dot \Delta \dot P + \dot \Delta \dot Q Q^{-1} - \dot Q Q^{-1} P \dot \Delta - \dot P P \dot \Delta) \right) Q.
\end{align*}
Here, $\dot Q$ is shorthand for $\frac{d}{dt}|_{t=0} Q(g_t)$, and analogously for $\pi, P, \Delta$ and $S$. To obtain the last line, we used that $\dot S P + P \dot P P = 0$ by the defining relation of $S$. The resulting expression cannot necessarily be simplify very conveniently. However, if $\dot g$ is such that $( \dot \Delta u, v)_{L^2(M,g_0)} = 0$ for any $u,v \in E_\lambda(g_0)$, and hence, $P\dot \Delta P = 0$, then nearly all expressions cancel and we obtain, after taking inner products,
\begin{align*}
    \ddot \pi_{\alpha\beta}
    &= \left( (\ddot \Delta + 2 \dot \Delta \dot P) u_\alpha, u_\beta \right)_{L^2(M,g_0)}.
\end{align*}
 Since the curve $g_t$ lies in the conformal class $\mathcal G(M,g_0)$, we may write $g_t = f_t g_0$. By conformal covariance of the Laplacian in two dimension, $\Delta_{g_t} = f_t \Delta$. Hence,
 \begin{align*}
     A^{\alpha\beta} \ddot \pi_{\alpha\beta}
    &= A^{\alpha\beta} \left( (\ddot f \Delta + 2 \dot f \Delta \dot P) u_\alpha, u_\beta \right)_{L^2(M,g_0)} \\ &= A^{\alpha\beta} \int_M \ddot f A^{\alpha\beta} u_\alpha u_\beta \, d\mu_{g_0} + 2 A^{\alpha\beta}\left( \dot f \Delta \dot P u_\alpha, u_\beta \right)_{L^2(M,g_0)} \\
    &= 2 \lambda A^{\alpha\beta}\left(\dot f (\Delta \circ L \circ \dot f \circ \Delta) u_\alpha, u_\beta \right)_{L^2(M,g_0)}.
 \end{align*}
 To obtain the last line, we used $A^{\alpha\beta} u_\alpha u_\beta = 0$ and the identity $\dot P = L \circ \dot \Delta$, where $L$ satisfies $(\lambda - \Delta) \circ L = \mathbb I - P$. This identity lets us reduce the expression above further.
 \begin{align*}
     A^{\alpha\beta} \ddot \pi_{\alpha\beta}
    &= 2 \lambda A^{\alpha\beta}\left(\dot f (\lambda L - \mathbb I + P) (\dot f u_\alpha), u_\beta \right)_{L^2(M,g_0)} \\
    &= 2 \lambda^2 A^{\alpha\beta} \left(\dot f L (\dot f u_\alpha), u_\beta \right)_{L^2(M,g_0)} \\
    &= 2\lambda^2 \int_{M} \dot f(x) \dot f(y) L(x,y) A^{\alpha\beta} u_{\alpha}(x) u_{\beta}(y).
 \end{align*}
 
To establish indefiniteness of $A^{\alpha\beta} D^2\pi(g_0)_{\alpha\beta}$ on $K$, it suffices to show $A^{\alpha\beta} \ddot \pi_{\alpha\beta}$ is nonzero for some curve $g_t$ in $\mathcal U$ tangent to $K$. This follows from the existence of the involution $\Phi$ as assumed in the statement of this lemma. The kernel $L$ is invariant under pull-back by the isometry $\Phi$, while the term $A^{\alpha\beta} u_{\alpha}(x) u_{\beta}(y)$ changes sign. Hence, $A^{\alpha\beta} \ddot \pi_{\alpha\beta} > 0$ along a curve $g_t$ implies that $A^{\alpha\beta} \ddot \pi_{\alpha\beta} < 0$ along the curve $\Phi^\ast g_t$.

Suppose for the sake of contradiction that $A^{\alpha\beta} \ddot \pi_{\alpha\beta} = 0$ along all curves $g_t$ with $\dot g \in K$. By polarizing, we find that
\begin{align*}
    \int_{M} f_1(x) f_2(y) L(x,y) A^{\alpha\beta} u_{\alpha}(x) u_{\beta}(y) = 0,
\end{align*}
for all $f_1,f_2 \in K$ (identifying $K \in T_{g_0} \mathcal U$ with $\{f \in C^\infty(M): f g_0 \in K\}$ for the sake of simplicity). By density of pure tensors, and since $L(x,y) \in L^2(M,g_0)$, $A^{\alpha\beta} u_{\alpha}(x) u_{\beta}(y) \in C^\infty(M)$, $L(x,y) A^{\alpha\beta} u_{\alpha}(x) u_{\beta}(y)$ is contained in the orthogonal complement of $K \otimes K$ in $L^2(M,g_0)^{\otimes 2}$, i.e. in $K^\bot \otimes K^\bot \oplus K^\bot \otimes K \oplus K \otimes K^\bot$. Since $K^\bot \subseteq C^\infty(M)$, the wave-front set of $L(x,y) A^{\alpha\beta} u_{\alpha}(x) u_{\beta}(y)$ at any point $(x,y) \in M\otimes M$ lies in the set $T_x M \oplus \{0\} \cap \{0\} \oplus T_y M$ \cite[Theorem 8.2.9]{HormanderI}. However, since $L(x,y) A^{\alpha\beta} u_{\alpha}(x) u_{\beta}(y)$ is the Schwartz kernel of the pseudodifferential operator $A^{\alpha\beta} (u_\alpha \circ L \circ u_\beta)$, its wave-front set is also constrained to lie in the conormal bundle of the diagonal in $M\times M$, i.e., in $\{(x,y,\xi,\zeta): x=y, \xi = -\zeta \}$  \cite[Theorem 18.1.16]{HormanderIII}, and hence $L(x,y) A^{\alpha\beta} u_{\alpha}(x) u_{\beta}(y) \in C^\infty(M \times M)$. Since $L(x,y) = \frac{2}{\pi} \log(\frac{|x-y|}{2}) + \mathcal O(1)$ near the diagonal (see the proof of \Cref{lemma: smooth kernel}), and $A^{\alpha\beta} u_{\alpha}(x) u_{\beta}(y)$ does not vanish to infinite order anywhere, we arrive at the same contradiction as in the proof of \Cref{prop:generalizednontransverseMetric}.

Hence, $\mathcal S_m$ has finite codimension. By \Cref{theorem:nontransverse}, which says that metrics admitting an unstable eigenvalue form a set of infinite codimension, $\mathcal S_m$ must contain $g_1$ arbitrarily close to $g_0$ such that $\pi(g_1) = \lambda \mathrm{id}$ and $\lambda$ is a \emph{stable} eigenvalue of $\Delta_{g_1}$. By rescaling the metric, we may assume $\lambda = \lambda(g_0)$. Since transversal intersections are stable under $C^0$-perturbations, any sufficiently small $C^0$-perturbation $\tilde \pi$ of $\pi$ will take the value $\lambda \mathrm{id}$ somewhere, proving weak stability. Furthermore, the \tcodim\ of $\mathcal S_m$ near $g_1$ is exactly $\frac{m(m+1)}{2}-1$ by \Cref{theorem:transverse structure theorem}. By \Cref{theorem:nonCrossing}, the codimension of $\mathcal S_m$ near any other point of $\mathcal U$ is at most $\frac{m(m+1)}{2}-1$, proving the second claim.
\end{proof}

\begin{proof}[Proof of \Cref{ex:nontransverse_structure}]
    Let $m=4$, and let $\mathcal U \subseteq \mathcal G(\mathbb T^2,g_0)$, $\pi: \mathcal U \to \mathbb R^{\frac{m(m+1)}{2}}$ and $\mathcal S_4$ be as in the proof of \Cref{ex:WAH_2d}. Suppose for the sake of contradiction that $\mathcal S_4$ was a submanifold of $\mathcal U$, which by \Cref{ex:WAH_2d} must be of codimension $9$. By first-order perturbation theory (as detailed in \Cref{theorem: general local structure theorem}), $T_{g_0} \mathcal S_4$ is contained in the space 
    \begin{align*}
        N &= \mathrm{span}(\{u_{\alpha}u_{\beta}: 1\leq \alpha \leq \beta\leq 4\})^\bot \oplus \mathrm{span}(1) \\
        &= \mathrm{span}\left(\begin{smallmatrix}\cos\\\sin\end{smallmatrix}(2x),\begin{smallmatrix}\cos\\\sin\end{smallmatrix}(2y), \begin{smallmatrix}\cos\\\sin\end{smallmatrix}(x+y),\begin{smallmatrix}\cos\\\sin\end{smallmatrix} (x-y)\right)^\bot
    \end{align*}
    The action of $\mathbb T^2$ on itself by translation descends to an action on $\mathcal G(\mathbb T^2,g_0)$ via pull-back, and $\mathcal S_4$ is invariant under this action since the spectrum of the Laplacian $\Delta_g$ is. Since $g_0$ is a fixed point, $T_{g_0} \mathcal S_4 \subseteq T_{g_0} \mathcal G(\mathbb T^2,g_0) \cong \mathcal C^\infty(\mathbb T^2)$ would have to be translation-invariant as well. Since $(T_{g_0} \mathcal S_4)^\bot$ is a real subspace of $C^\infty(\mathbb T^2)$, it splits into a sum of two-dimensional irreducibles $\mathrm{span}(\cos(kx +\ell y),\sin(kx +\ell y))$, $k,\ell \in \mathbb Z$ and possibly $\mathrm{span}(1)$, but the latter is impossible since $\mathrm{span}(1) \subseteq T_{g_0} \mathcal S_4$ by the scaling invariance of $\mathcal S_4$. This contradicts the assumption that $T_{g_0} \mathcal S_4$ is of codimension $9$.
\end{proof}

\subsection{An example with no unstable eigenvalues: The round $2$-sphere}
\label{sec:sphere}

A homogeneous space exhibiting very different behaviour is $\mathbb S^2$, with its standard metric.
Its eigenvalues $0 = \lambda_0 > \lambda_1 > \lambda_2 > \ldots$ are given by $\lambda_\ell = -\ell(\ell+1)$, $\ell \in \mathbb N$,  and have multiplicity $2\ell+1$. The corresponding eigenspace comprises all homogeneous harmonic polynomials of degree $\ell$, restricted to $\mathbb S^2$.
It may seem, due to the high multiplicities occurring in this example, that some or all of these eigenvalues could be unstable. However, this is not the case, as was shown by Colin de Verdi\`ere \cite{CdV1988}. We reproduce the proof here because it is instrumental in understanding the higher dimensional case, which was not discussed in \cite{CdV1988}.

\begin{proposition}[{\cite[pg.\ 186]{CdV1988}}]
\label{prop: sphere}
    All nonzero eigenvalues of the Laplace-Beltrami operator on $\mathbb S^2$ with its standard metric are conformally stable.
\end{proposition}

This statement is related to the Clebsch-Gordan theory of $SO(3)$, for which see \cite[Appendix C]{HallRepTheory}. However, the proof only uses basic properties of spherical harmonics and elementary representation theory. A particularly neat overview of all necessary facts is given in \cite[Chapter 7]{KosmannSchwarzbach}. 

\begin{proof}
    Consider the $\ell^{th}$ eigenvalue, $\lambda = -\ell(\ell+1)$, of the Laplace-Beltrami operator on $\mathbb S^2$.
    The isometry group of the sphere, $SO(3)$, is irreducibly represented on the corresponding eigenspace $E_{\lambda}$ via $g \cdot u = (g^{-1})^\ast u$. Denote the set of harmonic homogeneous polynomials of degree $\ell$ in three variables by $\mathcal H_\ell$, so that $E_{\lambda} = \mathcal H_\ell \rvert_{\mathbb S^2}$.
    
    Consider the set of even homogeneous polynomials of degree $2\ell$, which we shall call $\mathcal E_{2\ell}$. It is likewise a representation of $SO(3)$, which decomposes into 
    \begin{align*}
        \mathcal E_{2\ell}(\mathbb R^3) = \oplus_{j=0}^\ell \lVert x \rVert^{2(\ell-j)} \mathcal H_{2j} \cong \oplus_{j=0}^\ell \mathcal H_{2j}
    \end{align*}
    The product of two eigenfunctions $u,v \in E_\lambda$ is an element of $\mathcal E_{2\ell} \rvert_{\mathbb S^2}$. This suggests to consider the linear map $\Phi: \vee^2 \mathcal H_\ell \to \mathcal E_{2\ell}(\mathbb R^3)$ which sends $\frac{1}{2} \left( u \otimes v + v \otimes u \right)$ to the polynomial $u(x) v(x)$. Conformal nondegeneracy of $\lambda$ is equivalent to $\ker \Phi = \{0\}$. A quick dimension count reveals that $\dim \vee^2 \mathcal H_\ell = (\ell+1)(2\ell+1) = \dim \mathcal E_{2\ell}$, hence $\ker \Phi = \{0\}$ if and only if $\Phi$ is surjective.
    
    By Schur's Lemma it is sufficient to show that each irreducible component $\mathcal H_{2j}$ of $\mathcal E_{2\ell}(\mathbb R^3)$ intersects $\img\Phi$ nontrivially. To show this, consider the subspace $\mathcal Z \subseteq \mathcal E_{2\ell}$ of polynomials invariant under rotation around the $z$-axis. It has dimension $\ell+1$ and intersects each $\mathcal H_{2j}$ in a one-dimensional space generated by the ``zonal'' spherical harmonic of degree $2j$. Thus, it suffices to show $\mathcal Z \subseteq \img \Phi$.

    One basis of $\mathcal H_\ell$ is given by \emph{Laplace's spherical harmonics}: For $m \in \{0,\ldots,\ell\}$, let $Y^m_\ell(x,y,z) = c(m,\ell) (x+iy)^m Q^m_\ell(z)$, where $Q^m_\ell(z)$ is the respective associated Legendre polynomial, divided by $(1-z^2)^{\frac{m}{2}}$, and $c(m,\ell)$ is a normalization constant. For $m \in \{-\ell\ldots,0\}$, instead $Y^m_\ell(x,y,z) = c(|m|,\ell) (x-iy)^{|m|} Q^{|m|}_\ell(z)$. The products
    \begin{align*}
        Y^m_\ell(x,y,z)Y^{-m}_\ell(x,y,z) = c(m,\ell)^2 (x^2+y^2)^m Q^m_\ell(z)^2
    \end{align*}
    are linearly independent as $m$ ranges from $0$ to $\ell$, since they are all of distinct degree in $z$. Because these are $\ell+1$ linearly independent functions in $\mathcal Z \cap \img \Phi$, we conclude that $\mathcal Z \subseteq \img \Phi$ as desired.
\end{proof}

The situation changes drastically in higher dimensions. On $\mathbb S^n$ with its standard metric, the eigenvalues of the Laplacian are $\lambda_\ell = -\ell(\ell+n-1)$, $\ell \in \mathbb N$,  and have multiplicity 
\begin{align*}
    \binom{n+\ell}{n} - \binom{n+\ell-2}{n}.
\end{align*}
The proof of \Cref{prop: sphere} could be repeated under some minor changes, if not for one crucial issue:
The dimension of $\vee^2 \mathcal H_\ell$ is a polynomial of degree $2(n-1)$ in $\ell$. On the other hand $\dim \mathcal E_{2\ell}$ is merely of degree $n$ in $\ell$. Hence, if $n \geq 3$, the map $\Phi: \vee^2 \mathcal H_\ell \to \mathcal E_{2\ell}$, which assigns the polynomial $uv \in E_{2\ell}$ to the tensor $\frac{1}{2} \left( u \otimes v + v \otimes u \right)$, must have a nontrivial kernel for all large enough $\ell$. For these $\ell$, the eigenvalue $\lambda_\ell$ is conformally unstable. A slightly more involved dimension comparison shows that $\lambda_\ell$ is in fact unstable for all large enough $\ell$.

Let us obtain a precise upper bound for the degree of the smallest unstable eigenvalue on $\mathbb S^3$. Denote the spaces of homogeneous polynomials in $4$ variables of degree $\ell$ by $\mathcal E_{\ell}$, and the space of harmonic homogeneous polynomials by $\mathcal H_{\ell}$. Consider the map $\vee^2 \mathcal H_\ell \to \mathcal E_{2\ell-2}^{\frac{4\times 5}{2}}$ which sends $\frac{1}{2} \left( u \otimes v + v \otimes u\right)$ to $\left(\frac{du}{dx^{j}}(x) \frac{dv}{dx^k}(x) \right)_{j,k=1}^{4}$. By \Cref{proposition: further reduction}, degeneracy of $\lambda_\ell$ is equivalent to this map having a nontrivial kernel, which is certainly the case if 
\begin{align*}
    &\dim \vee^2 \mathcal H_\ell \geq 10 \dim \mathcal E_{2\ell-2},
\end{align*}
i.e., if the following inequality of binomial coefficients holds:
\begin{align*}
    &\frac{1}{2}\left( \binom{3+\ell}{3} - \binom{1+\ell}{3} + 1\right) \left( \binom{3+\ell}{3} - \binom{1+\ell}{3}\right)
    \geq 10 \binom{1+2\ell}{3},
\end{align*}
For $\ell > 1$, this inequality reduces to
\begin{align*}
    \ell^3 - \frac{65}{3}\ell^2 - \frac{44}{3}\ell -2 \geq 0,
\end{align*}
which is satisfied when $\ell \geq 23$ (both reduced and solved via \emph{Mathematica} \cite{Mathematica}). Thus, any eigenvalue of degree at least $23$ of $\mathbb S^3$ is unstable. Since any solution to the degeneracy equation in four variables lifts to one in $n \geq 4$ variables, we conclude that the same holds for all $\mathbb S^n$, $n\geq 3$.

\subsection{Further examples and observations}
\label{sec:misc}

\subsubsection{Products of homogeneous spaces}

Let $M$ be a connected, closed Riemannian manifold of volume $1$ such that its isometry group, $\mathrm{Aut}(G)$, acts transitively on $M$. On any eigenspace $E_\lambda$ of its Laplace-Beltrami operator, the isometry group acts by orthogonal transformations. A well known consequence (see e.g., \cite[Corollary 2.3]{ElSoufi2003}) is that $\sum_{\alpha=1}^m u_\alpha^2 = m$ for any real valued orthonormal basis $\{u_\alpha\}_{\alpha=1}^m$ of $E_\lambda$.

Thus, if $\lambda < 0$ is a Laplace-Beltrami eigenvalue for two such homogeneous spaces $M$ and $M'$, then $\lambda$ becomes a conformally unstable eigenvalue on the product $M \times M'$. After all, given orthonormal bases $\{u_\alpha\}_{\alpha=1}^m$ and $\{v_\alpha\}_{\alpha=1}^{m'}$ of the corresponding eigenspaces on $M$ and $M'$, then $u_1 \otimes 1, \ldots, u_m \otimes 1, 1 \otimes v_1, \ldots, 1 \otimes v_{m'}$ is an orthonormal set of eigenfunctions on $M \times M'$, and hence,
$$\frac{1}{m} \sum_{\alpha = 1}^m (u_\alpha \otimes 1)^2 - \frac{1}{m'} \sum_{\alpha = 1}^{m'} (1 \otimes v_\alpha)^2 = 0$$
implies that $\lambda$ is conformally unstable, by \Cref{proposition:submersion equivalent linear nondeg}. However, it need not be unstable, as the example of the torus shows (see \Cref{ex:nontransverse}).

\subsubsection{Disconnected Manifolds}
\label{subsubsection:disconnected manifolds}
It is easily seen, but instructive that connectedness of $M$ is absolutely necessary for \Cref{theorem:nonCrossing} to hold. To see this, suppose $M$ is the disjoint union of closed Riemannian manifolds $M_1$ and $M_2$ such that $\lambda>0$ is a Laplace-Beltrami eigenvalue of both $M_1$ and $M_2$ with corresponding eigenfunctions $u_1$ and $u_2$. Extended by zero on the other component, these functions are also eigenfunctions on $M$. Since $(u_1 - u_2)^2 = (u_1 + u_2)^2$ and $(du_1 - du_2)^2 = (du_1 + du_2)^2$, $\lambda$ is an unstable eigenvalue of $M$. In this way, an unstable eigenvalue on $M$ arises every time an eigenvalue on $M_1$ coincides with an eigenvalue on $M_2$. In particular, a crossing of simple eigenvalues on $M_1$ and $M_2$, respectively, leads to a smooth hypersurface in $\mathcal G(M)$ of metrics admitting an unstable eigenvalue of multiplicity $2$.

Connectedness of $M$ is used in the proof of \Cref{theorem:nonCrossing} exactly once, to conclude via unique continuation that, for an orthonormal basis $\{u_\alpha\}_{\alpha=1}^m$ of $E_\lambda$ and nonzero $A \in \mathbb R^{\frac{m(m+1)}{2}}$, the function $A^{\alpha\beta}u_\alpha(x)u_\beta(y)$ does not vanish to infinite order along the diagonal in $M\times M$. The failure of \Cref{theorem:nonCrossing} for disconnected $M$ stems from the failure of unique continuation for eigenfunctions in this setting.

\subsubsection{Comparison with the Hodge Laplacian}
\label{subsubsection:comparison with the Hodge Laplacian}
Comparing the unstable eigenvalues of the Laplace-Beltrami operator to those of the Hodge Laplacian is particularly interesting. We call an eigenvalue of the Hodge Laplacian unstable it does not satisfy the SAH. This has a corresponding ``localized'' form analogous to the one derived for the Laplace-Beltrami operator in Proposition \ref{proposition:submersion equivalent linear nondeg}. It was shown in \cite{Kepplinger2022} that the spectrum of the $3$-dimensional Hodge Laplacian in degree $1$ has a more complicated structure than that of the Laplace-Beltrami operator: Eigenvalues corresponding to the exact, positive coexact, and negative coexact spectrum may collide and form unstable eigenvalues. Indeed it is proven that metrics with unstable eigenvalues are a codimension $1$ occurence, that is one cannot hope to avoid them (or double eigenvalues more generally) along $1$-parameter families of Riemannian metrics. In particular, the lowest eigenvalue $\lambda=4$ of the Hodge Laplacian restricted to the coexact spectrum on the round $3$-sphere has multiplicity $6$ and is unstable. \par
This degeneracy is resolved by passing to the curl operator, the square root of the Hodge Laplacian on the coexact spectrum. It is not hard to prove that unstable eigenvalues for the curl operator must have multiplicity at least $4$. This implies that the two curl eigenvalues of lowest absolute value on the round $3$-sphere are stable, as both have multiplicity $3$. It is worth pointing out that conformally unstable eigenvalues of the curl operator occur already in multiplicity $2$, for example in Berger's sphere.

\subsubsection{Decomposability}
The observations \ref{subsubsection:disconnected manifolds} and \ref{subsubsection:comparison with the Hodge Laplacian} in this subsection show that Arnold's transversality hypothesis or a maximal non-crossing rule do not always hold even for very large families of operators. Some thought reveals that all of the aforementioned situations have a common problem: the family of operators is in a certain sense decomposable. More precisely, there exist smooth families of projections
\begin{align*}
    \pi_i (q): H \to H_i (q)
\end{align*}
which commute with the operator $A(q)$ for all $q \in \mathcal{X}$, where $H_i$ are nontrivial closed subspaces so that $H= \bigoplus_{i} H_i (q)$. For the Laplace-Beltrami operators on disconnected manifolds these projections are given naturally by the restriction to any of the connected components, in the case of the Hodge Laplacian we get (metric dependent) projections to the closed subspaces associated with the exact, positive coexact, and negative coexact spectrum. If the spectra of the $A(q)$ restricted to the closed subspaces $H(q)$ behave sufficiently independently one should expect collisions of eigenvalues coming from different $H(q)$.

\subsubsection{Equivariance}
There is a different but related issue regarding generic simplicity of the spectrum. Sometimes a family of operators is obstructed from having simple spectrum because of an internal symmetry of the operators themselves, i.e., if there exists an action on the domain of the operator which commutes with the operator. The simplest example known to us is the Hodge Laplacian on $1$-forms in dimension $2$ where eigenvalues coming from functions and those coming from $2$-forms coincide, independent of the choice of metric. This is a consequence of the fact that the Hodge Laplacian commutes with the Hodge star operator which is an almost complex structure on the space of $1$-forms in dimension $2$.

This is a more general effect in dimension $2n$ for forms of degree $n$, see the work of Millman \cite{Millman1980}. Interestingly, a similar phenomenon occurs for the middle degree $2n$ in dimensions $4n+1$ as shown by Gier and Hislop \cite[Section 3.1]{GierHislop2016}.

\emph{Thus, it seems that equivariance with respect to some group action is an obstruction to generic simplicity, and decomposability one to generic simplicity along $1$-parameter families.}

Arnold exhibited an example of a family of operators in which both phenomena occur at the same time \cite[Appendix 10, subsection D]{Arnold1997MathMethods}.
He pointed out that Dirichlet eigenvalues of the Laplacian on domains with a $\frac{1}{3}$-rotational symmetry come in two types: those of multiplicity $1$, whose associated eigenfunction is symmetric, and those of multiplicity $2$, whose associated eigenfunctions are antisymmetric with respect to the $\mathbb Z_3$ action. By varying the domain while preserving the symmetry one can achieve that symmetric and antisymmetric eigenvalues collide (or pass through one another), but eigenvalues of the same type will never merge. The equivariance of this family of operators leads to non-simple spectrum, and the decomposability, given by projections to the symmetric and antisymmetric eigenspaces, leads to collisions among these different types of eigenvalues. It is worth pointing out that there are no more collision events once one restricts to either symmetric or antisymmetric eigenvalues.

\subsection{Open questions}

We conclude our paper with a list of natural questions, and some ideas for how one might resolve them.

\begin{question}
    \label{q:multiplicity7}
    Do there exist unstable eigenvalues of multiplicity $7$?
\end{question}

\Cref{theorem: nonexistence of low multiplicity nontransverse eigenvalues} leaves an interesting gap: Unstable eigenvalues cannot have multiplicity less than seven, but all examples known to the authors have multiplicity at least eight. In the proof of \Cref{theorem: nonexistence of low multiplicity nontransverse eigenvalues}, nonexistence of unstable eigenvalues of multiplicity six follows from the rigidity of local isometries of $\mathbb S^2$. Local isometric immersions of $\mathbb S^2$ into $\mathbb S^3$ are no longer fully rigid, but they do exhibit an interesting partial rigidity phenomenon, as was shown by O'Neill and Stiel \cite{ONeillStiel}: As in the case of local isometric immersions of $\mathbb R^2$ into $\mathbb R^3$, their image is ruled by geodesics. One might use this to show unstable eigenvalues of multiplicity seven do not exist, but this seems to go beyond the scope of the present paper. Let us also note that there exists a certain left-invariant metric on $SU(2) \cong \mathbb S^3$ admitting a first eigenvalue of multiplicity $7$ (as shown by Urakawa in \cite{urakawa}), and thus a plausible candidate for an unstable eigenvalue of this multiplicity. However, a closer examination reveals that this eigenvalue is stable.

\begin{question}
    What is the local behaviour of the spectrum near metrics with unstable eigenvalues?
\end{question}

According to \Cref{theorem: general local structure theorem}, given a metric $g$ admitting a stable eigenvalue of multiplicity $m$,
the set of perturbations which do not split up this eigenvalue is a smooth manifold of codimension $\frac{m(m+1)}{2}-1$. Understanding the behaviour of an \emph{unstable}
eigenvalue under perturbation is more difficult. The natural defining function $\pi$ from \Cref{section: local structure theorem} is no longer a submersion, but it could still cut out a submanifold (much like the function $y^2 = 0$ cuts out a submanifold of $\mathbb R^2_{x,y}$). However, this seems unlikely. It would be interesting, but potentially difficult, to exhibit a rigorously treated example.

The paper \cite{Arnold1972} contains a discussion of the complicated shape of the singularity of the set of symmetric matrices with an eigenvalue of multiplicity at least $2$ along the lower dimensional submanifold of symmetric matrices with an eigenvalue of multiplicity $3$ (Remark 1.4 in \cite{Arnold1972}). By \Cref{theorem: general local structure theorem}, near a stable eigenvalue of multiplicity $3$, the set of metrics splitting off exactly one eigenvalue inherits the shape of this singularity, up to a trivial factor of finite codimension. The shape of the tangent cone of the analogous singularity near an \emph{unstable} eigenvalue is presumably topologically different, but one would likely need to take at least a second variation to compute it exactly.

\begin{question}
    Does there exist a ``natural'' parameter space containing the Laplace-Beltrami operators with respect to which there are no unstable eigenvalues?
\end{question}

By enlarging the parameter space, more and more ``directions'' in which eigenvalues can split up become available. Thus, it is easy to find unstable eigenvalues with respect to variations in a fixed conformal class, but already harder if one allows variations among all possible metrics. A natural question is whether there exists a good parameter space containing all Laplace-Beltrami operators, e.g., elliptic differential operators of degree $2$, where all eigenvalues satisfy the strong Arnold hypothesis.

\begin{question}
\label{question:geometric meaning}
    What is the geometric meaning of the degeneracy condition?
\end{question}
As briefly mentioned in the introduction, the (conformal) degeneracy condition looks very similar to the one obtained by El Soufi and Ilias in \cite{ElSoufi2008} characterizing metrics which extremize the $k$-th eigenvalue functional. They show that a metric $g$ extremizes the eigenvalue functional $\lambda_k$ if there exists a basis of eigenfunctions $\{u_j\}_{j =1}^{m}$ associated the same eigenvalue and coefficients $\epsilon_j$ valued in $\{0,1\}$ satisfying
\begin{align}
\label{equation:extremality}
    \sum_{j} \epsilon_j \, du_j \! \otimes \! du_j = g
\end{align}
and the converse is true provided one considers the greatest or smallest of the eigenvalue functionals coinciding at $g$. The analogous extremality condition in a conformal class is
\begin{align*}
    \sum_{j} \epsilon_j u_j^2 = 1
\end{align*}
Using a result of Takahashi \cite{Takahashi1966} one can reinterpret the extremality condition as saying that the eigenfunctions $\{u_j\}$ give a minimal isometric immersion of $(M,g)$ into the sphere $\mathbb S^{d-1} (\sqrt{\frac{n}{|\lambda_k|}})$, where $d$ is the number of nonzero coefficients $\epsilon_j$. Conversely, every such minimal immersion gives rise to an extremal metric.

It would be interesting to see whether the nondegeneracy condition admits a similar interpretation, perhaps as a map with special geometric properties into the null cone of a semi-Riemannian $\mathbb{R}^{s,t}$.

We would also like to remark that, exactly as in the extremality condition, the condition on the square sum of the eigenfunctions in the definition of an unstable eigenvalue is in fact implied by the condition on the square sum of their gradients (\Cref{proposition: further reduction}).

\begin{question}
    What is the topology of the set of metrics which exhibit at least one eigenvalue of multiplicity $m$? Is it connected?
\end{question}


\appendix

\section{Cofinite Fr\'{e}chet submanifolds}
\label{section: appendix 0}

Our first appendix establishes the ``cofinite'' version of the implicit function theorem used throughout the text. It is analogous to \cite[Corollary 5.5]{Freyn2015}, but has a significantly simpler proof. Recall the definition of a $C^1$ map between Fr\'{e}chet spaces \cite[Definition 3.1.1.]{Hamilton1982}: A map $F: \mathcal F \supseteq \mathcal U \to \mathcal G$ is called \emph{continuously differentiable}, or just $C^1$, if 
\begin{align*}
    DF(x)u = \lim_{t\to 0} \frac{F(x + tu) - F(x)}{t}
\end{align*}
exists for all $x \in \mathcal U$, $u \in \mathcal F$, and furthermore, $DF(x)u$ is linear in $u$ and \emph{jointly continuous} in $x$ and $u$.

\begin{lemma}
\label{lemma: appendix: cofinite implicit function theorem}
    Let $\mathcal U \subseteq \mathcal F$ and $V \subseteq \mathbb R^m$ be open neighborhoods of the origin in a Fr\'{e}chet space $\mathcal F$ and in $\mathbb R^m$, respectively. Suppose $F:\mathcal U \times V \to \mathbb R^m$ is $C^1$ in the sense introduced above, $F(0,0) = 0$, and $D_2 F(0,0)$ is invertible. Then there exists a neighborhood $\mathcal U' \times V' \subseteq \mathcal U \times V$ of $(0,0)$ and a $C^1$ function $G:\mathcal U' \to \mathbb R^m$ such that for all $(x,y) \in \mathcal U' \times V'$, we have $F(x,y) = 0$ if and only if $y = G(x)$.
\end{lemma}

\begin{proof}
    This is a matter of checking that the usual fixed point iteration works under the present, weaker than usual, assumptions on $\mathcal F$. Define $R: \mathcal U \times V \to \mathbb R^m$, 
    \begin{align*}
        R(x,y) = y - D_2 F(0,0)^{-1} F(x,y).
    \end{align*}
    Clearly, $F(x,y) = 0$ if and only if $R(x,y) = y$. We can rewrite
    \begin{align*}
        R(x,y) = \left( \int_0^1 \left( \mathrm{id}_{\mathbb R^m} - D_2 F(0,0)^{-1} D_2 F(x,ty) \right) dt \right) y
    \end{align*}
    By assumption, the expression $D_2 F(x,y) h$ is jointly continuous in $x,y$ and $h$. Since we assumed $V$ is \emph{finite-dimensional}, this means that the operator $D_2 F(x,y):\mathbb R^m \to \mathbb R^m$ is continuous in $x,y$ with respect to the norm topology. Thus, there exists a neighborhood $\mathcal U' \times V' \subseteq \mathcal U \times V$ of $(0,0)$ such that 
    \begin{align*}
        \lVert \mathrm{id}_{\mathbb R^m} - D_2 F(0,0)^{-1} D_2 F(x,y) \rVert < \frac{1}{2}
    \end{align*}
    for all $x,y \in \mathcal U' \times V'$, and hence $R(\mathcal U' \times V') \subseteq V'$. Analogously, we have that $R(x,y)$ is a contraction on $V'$ for any fixed $x \in \mathcal U'$. Indeed, $R(x,y_1)-R(x,y_2)$
    \begin{align*}
        = &\left( \int_0^1 \left( \mathrm{id}_{\mathbb R^m} - D_2 F(0,0)^{-1} D_2 F(x,(1-t)y_1 + ty_2) \right) dt \right) (y_2-y_1).
    \end{align*}
    Thus, given $x \in \mathcal U'$, there is a unique solution $y \in V'$ to the equation $F(x,y) = 0$. 

    To prove that $F$ admits all directional derivatives is straightforward. Let $x \in \mathcal U'$, and $u \in \mathcal F$. Then the finite dimensional implicit function theorem, applied to the restriction of $F$ to $\left\langle x, u \right\rangle \times \mathbb R^m$ shows that $DG(x)[u] = \lim_{t\to 0} \frac{G(x+tu) - G(x)}{t}$ exists and
    \begin{align*}
        DG(x)[u] = D_yF(x,G(x))^{-1} D_x F(x,G(x)) u.
    \end{align*}
    This expression is clearly continuous in $x$ and $u$, and hence, $G$ is $C^1$.
\end{proof}

\begin{proposition}
    \label{proposition: appendix: submersion}
    Let $\mathcal U \subseteq \mathcal F$ be an open subset of a Fr\'{e}chet space $\mathcal F$. Suppose a $C^1$ map $F: \mathcal U \to \mathbb R^m$ is such that $DF(p): \mathcal F \to \mathbb R^m$ is surjective at a point $p \in \mathcal U \cap F^{-1}(0)$. Then there exists a neighborhood $\mathcal U' \subseteq \mathcal U$ of $p$, an open neighborhood $\mathcal V \subseteq \mathcal F'$ of the origin in a Fr\'{e}chet space $\mathcal F'$, and a $C^1$ map $G:\mathcal V \to \mathcal U'$ such that $F^{-1}(0) \cap \mathcal U'= G(\mathcal V)$. If $\mathcal F$ was equipped with a grading, $G$ is furthermore tame.
\end{proposition}

\begin{proof}
    Since $DF(p)$ is onto, we may choose $v_1,\ldots,v_m \in \mathcal F$ such that $\{DF(p)v_j\}_{j=1}^m$ form a basis of $\mathbb R^m$. By continuity of $q \mapsto DF(q) v_j$, there exists a neighborhood $\mathcal U_1$ of $p$ such that $DF(q)v_1,\ldots, DF(q)v_m$ are linearly independent for all $q \in \mathcal U_1$. Note that $\mathcal F' := \ker DF(p)$ is a closed complement to $V := \left\langle v_1,\ldots,v_m \right\rangle$. Let $\mathcal V_1 \subseteq \ker DF(p)$ and $U_1 \subseteq V$ be open neighborhoods of the origin with $p + \mathcal V_1 \oplus U_1 \subseteq \mathcal U_1$. The map $\tilde F: \mathcal V_1 \times U_1 \to \mathbb R^m$, $\tilde F(x,y) = F(p + x + y)$ satisfies the hypothesis of \Cref{lemma: appendix: cofinite implicit function theorem}, which yields a neighborhood $\mathcal V \oplus U \subseteq \mathcal V_1 \oplus U_1$ of the origin and a $C^1$ map $\tilde G:\mathcal V \to U$ such that $F^{-1}(0) \cap (p + \mathcal V \oplus U)$ is parametrized by $G: \mathcal V \to \mathcal F$, 
    \begin{align*}
        G(x) = p + x + \tilde G(x).
    \end{align*}
    The first claim follows after we denote $\mathcal U' := p + \mathcal V \oplus U$.
    
    The map $G$ is tame with respect to any grading on $\mathcal F$ (and the induced grading on $\mathcal F'$), since it is the sum of an isometric embedding and a map of finite rank.
\end{proof}

A statement akin to \Cref{proposition: appendix: submersion} for submersions valued in a Banach space is obtained in \cite{Freyn2015}, albeit with the added complication that surjectivity of $DF$ at a single point no longer implies $F$ is a submersion. This theorem implies \Cref{proposition: appendix: submersion} immediately (Corollary 5.5 in \cite{Freyn2015}), but its proof invokes the Nash-Moser implicit function theorem. We found it noteworthy that cofinite Fr\'{e}chet submanifolds can be parametrized using the standard implicit function theorem.

\section{Transversality theory in finite codimension}
\label{section: appendix 1}
The aim of this appendix is to provide the necessary transversality-theoretic background needed throughout the paper.

\begin{definition}
A subset $\mathcal{D}$ of a Fr\'echet manifold $\mathcal{X}$ is said to have \tcodim \ $k$ if for any manifold $N$ of dimension less than $k$, the set of maps in $C^\infty (N,\mathcal{X})$ that avoids $\mathcal{D}$ is residual.
\end{definition}
It is clear from the definition that countable unions of sets of \tcodim \ $k_i$ have \tcodim \ $\min\limits_{i} k_i$.\\
Before we provide a nontrivial example of a set of \tcodim \ $k$ we will first need to prove a variant of the parametric transversality theorem. Due to the fact that we consider finite codimension manifolds only, its proof is essentially the same as in finite dimensions (see for example \cite{GuilleminPollack}), but for completeness we will include it here.

\begin{lemma}{[Parametric Transversality]}
\label{lemma: appendix: parametric transversality}
    Let $N$ be a smooth compact manifold of dimension $n$, $\mathcal{X}$ a Fr\'echet manifold, and $Z\subset \mathcal{X}$ the level set of a $C^1$ submersion $\phi:\mathcal{X}\to \mathbb{R}^k$, where $n\leq k$. Given a finite dimensional manifold $S$ of dimension $s$ and a smooth map $F:N\times S \to \mathcal{X}$ such that both $F$ and $\partial F$ are transverse to $Z$, $F_s$ and $\partial F_s$ are transverse to $Z$ for generic $s$. 
\end{lemma}
\begin{proof}
    Using the chain rule and that $\phi$ is a submersion we see that $W:=F^{-1}(Z)$ is a $C^1$ submanifold of $N\times S$ with boundary $\partial W=W\cap \partial(N \times S)$. Denote by $\pi:N\times S\to S$ the projection to the second factor. The conclusion of this Lemma will follow by an application of Sard's theorem once we prove that $s\in S$ being a regular value for the projection $\pi$ restricted to $W$ and $\partial W$ implies that $F_s\transv Z$ and $\partial F_s\transv Z$, respectively. Sard's theorem for $C^1$ maps holds provided the dimension of the source is less than or equal to the dimension of the target, which is where the dimensional restriction in the statement of the Lemma comes from.\\
    Assume now that $s$ is a regular value of $\pi$ restricted to $W$ (the boundary case works in exactly the same way). For $x\in W$, denote $F_s (x)=z\in Z$. Transversality of $F$ implies that
    \begin{align*}
        \mathrm{Im}(T F)_{(x,s)}+T_{z} Z=T_z \mathcal{X}
    \end{align*}
This means that for any $a\in T_z \mathcal{X}$ there exists $(b,c)\in T_{(x,s)} \big(N\times S\big)$ such that $(T F)_{(x,s)} (b,c)=a$. We need to find $v\in T_s S$ such that $F_s (v)=a+d$ for some $d\in T_z Z$. To this end note that the fact that $s$ is regular guarantees the existence of a $u\in T_{(x,s)} W$ so that $T \pi_s (u)=c$. Noting that $(T F)(u,e)\in T_z Z$ we set $v=b-u$ and compute
\begin{align*}
    (T F_s)_{(x,s)} (v)= (T F)_{(x,s)} (v,e)=a-(T F)(u,e)
\end{align*}
    and so $(T F_s)$ hits $a$ up to an element in $T_z Z$ and we are done.
\end{proof}
\begin{remark}
\begin{enumerate}
    \item For the purposes of proving a maximal non-crossing rule, the above Lemma is sufficient. If one wishes to control generic intersections of images of smooth maps from $N$ to $\mathcal{X}$ with $Z$ when $n>k$ via Sard's theorem one would have to prove smooth versions of Proposition \ref{proposition: appendix: submersion} and Lemma \ref{lemma: appendix: C1 dependence of laplacian}.
    \item The above proof clearly works when $M$ has corners.
    \item It also works if we prescribe the boundary $\partial M$ and only vary $F$ on the interior of $M$.
\end{enumerate}  
\end{remark}
The following Lemma tells us that \tcodim \ is a local concept in separable Fr\'{e}chet spaces.
\begin{lemma}
\label{lemma: appendix: local codimension implies global codimension}
Let $\mathcal{D}$ be a subset of a separable Fr\'{e}chet manifold $\mathcal{X}$. If for all $x\in \mathcal D$ there exists an open set $U_x \subset \mathcal{X}$ such that $\mathcal{D}\cap U_x$ has \tcodim \ $k$ in $\mathcal{X}$, then $\mathcal{D}$ has \tcodim \ $k$ in $\mathcal{X}$. 
\end{lemma}
\begin{proof}
    The union $\mathcal{O}=\bigcup_{x\in \mathcal D} U_x$ is an open set in $\mathcal X$ and therefore itself a separable Fr\'{e}chet manifold, so one may extract a countable subcover $\{U_{x_n}\}$ of $\mathcal{O}$. We conclude by noting that $\mathcal{D}$ is the countable union of the sets $U_n \cap \mathcal{D}$ all of which have \tcodim \ $k$.
\end{proof}

\begin{lemma}
\label{lemma: appendix: local transversality}
Let $\mathcal{X}$ be an open subset of a separable Fr\'echet space. Let $\mathcal{D}\subset \mathcal{X}$ be a subset for which there exist an open cover $(U_i)_{i\in I}$ and submersions $\phi_j:U_i\to \mathbb{R}^k$ with the property that, for each $p \in \mathcal D$, there is a neighborhood $V$ of $p$ and a countable set $J \subseteq I$ with $\mathcal D \cap V \subseteq \bigcup_{j \in J} \phi_j^{-1}(\{0\})$. Then $\mathcal{D}$ has \tcodim \ $k$. 
\end{lemma}

\begin{proof}
We need to prove that, for any manifold $N$ of dimension less that $k$, the set of smooth maps $C^\infty (N,\mathcal{X})$ from $N$ to $\mathcal{X}$ which avoids $\mathcal{D}$ is residual. Assume without loss of generality that $N$ is compact, as the general case follows by a compact exhaustion argument. We will show that a generic set in $C^\infty (N,\mathcal{X})$ avoids $\phi_j^{-1} (\{0\})$ for any $j$. The strategy is to show that every point $z\in \phi_j^{-1} (\{0\})$ has a neighbourhood $V_z$ so that $V_z \cap \phi_j^{-1} (\{0\})$ has \tcodim \ $k$. The claim then follows by an application of Lemma \ref{lemma: appendix: local codimension implies global codimension}. Let $B_z$ be some neighbourhood  of $z\in \phi_j^{-1} (\{0\})$ and $f\in C^\infty (N,\mathcal{X})$ which is not transverse to $B_z\cap \phi_j^{-1}(\{0\})$ at the point $z$. We define $F_z : M\times \mathbb{R}^k\to \mathcal{X}$ by $F_z (x,s)=f(x)+s_i h^i (f(x))$ where $f(x)=z$ and $h^i (f(x))$ a basis of a subspace $K_x \transv \ker(T\phi_j)$. After potentially restricting $B_z$ we may assume that $F_z$ is transverse to $B_z \cap \phi_j^{-1}(\{0\})$. The parametric transversality Lemma \ref{lemma: appendix: parametric transversality} now implies that the set of $s$ such that $F_s$ is transverse to $B_z \cap \mathcal{D}$ is residual which implies density in $\mathbb{R}^k$, meaning that $f$ can be approximated arbitrarily well in $C^\infty$ with smooth functions that are transverse to $B_z \cap \phi_j^{-1}(\{0\})$. 

Now pick an open subset $V_z$ of $B_z$ satisfying that $\overline{V_z \cap \phi_j^{-1} (\{0\})}\subset B_z$. Clearly, the set of maps $C^\infty (N,\mathcal{X})$ which avoids $\overline{V_z \cap \phi_j^{-1} (\{0\})}$ is open, and the above argument shows that it is dense as well. Since subsets of sets of \tcodim \ $k$ are themselves of \tcodim \ $k$, we have proven that $\phi_j^{-1}(\{0\})$ has \tcodim \ $k$. 

This immediately implies that $\mathcal{D}\cap V$ has \tcodim \ $k$, and another application of Lemma \ref{lemma: appendix: local codimension implies global codimension} allows us to conclude that $\mathcal{D}$ itself has \tcodim \ $k$.
\end{proof}

\begin{remark}
    The Fr\'echet manifold of smooth metrics is an open subset of the Fr\'echet space of smooth $(0,2)$-tensor fields, and so the setting of Lemma \ref{lemma: appendix: local transversality} is good enough for the purposes of \Cref{theorem:nonCrossing}. A more general setting in which one can run a similar argument is the class of separable Fr\'echet manifolds that admit smooth bump functions, i.e., those modelled on $C^\infty$-regular separable Fr\'echet spaces, see \cite[Chapter $3$, Subsection $14$]{MichorKrieglConvenientSetting}. In this case, one can define $F$ as above in a chart and extend it to a globally defined map via a bump function.
\end{remark}

\section{Continuous differentiability in the metric of the Laplacian and its spectral projections}
\label{section: appendix 2}

\begin{lemma}\label{lemma: appendix: C1 dependence of laplacian}
    Let $M$ be a closed, smooth manifold, and $\mathcal G(M)$ the Fr\'{e}chet manifold of smooth Riemannian metrics on $M$.
    \begin{enumerate}
        \item The family $\Delta: \mathcal G(M) \rightarrow \mathcal L(H^2(M),L^2(M))$ of Laplace-Beltrami operators is continuously differentiable. In local coordinates,
        \begin{align}
            D\Delta(g)[h] u = \tfrac12 g^{jk} \partial_j \tr_gh \hspace{1pt} \partial_k u - |g|^{-\frac12} \partial_j (|g|^{\frac12} h^{jk} \partial_k u).
        \end{align}
        \item Fix $g \in \mathcal G(M)$ and $\lambda \in \rho(\Delta_g)$. Then there exists a neighborhood $\mathcal U \subseteq  \mathcal G(M)$ of $g$ and $\varepsilon > 0$ such that $(z - \Delta)^{-1}: \mathbb B_\varepsilon(\lambda) \times \mathcal G(M) \rightarrow \mathcal L(L^2(M),H^2(M))$ is defined, and continuously differentiable.
        \item Fix $g \in \mathcal G(M)$ and $\gamma:\mathbb S^1 \rightarrow \rho(\Delta_g)$, a smooth, closed, positively oriented Jordan curve. For $\tilde g \in \mathcal G(M)$, let $ P_\gamma(\tilde g)$ denote the spectral projection to the eigenspaces of $\Delta_{\tilde g}$ corresponding to eigenvalues encircled by $\gamma$. Then there exists a neighborhood $\mathcal U \subseteq  \mathcal G(M)$ of $g$ such that $ P_\gamma: \mathcal U \rightarrow \mathcal L(L^2(M),H^2(M))$ is continuously differentiable.
    \end{enumerate}
\end{lemma}

\begin{proof}
    To allow computation in coordinates, choose a smooth partition of unity $\sum_{j=1}^\infty \chi_i = 1$ adapted to a coordinate atlas. It is then sufficient to check $\chi_i \Delta(g)$ is smooth in $g$ for all $i$. The operator $\chi_i \Delta(g)$ can be written as a sum of compositions of smooth vector fields $\partial_j$, $j=1,\ldots,n$ and multiplication operators which depend in a continuously differentiable fashion on $g$: 
    \begin{align*}
        \chi_i \Delta(g) = \chi_i \circ |g|^{-\frac{1}{2}} \circ \partial_j \circ |g|^{\frac{1}{2}} \circ g^{jk} \circ \partial_k
    \end{align*}
    This composition should be read as follows:
    \begin{enumerate}
        \item $\partial_k$ is a bounded linear operator $\mathcal L(H^2(M),H^1(M))$, independent of $g$,
        \item $|g|^{\frac{1}{2}} \circ g^{jk}$ is a $C^1$ family of bounded operators on $H^1(M)$,
        \item $\partial_j$ is a fixed bounded linear operator $\mathcal L(H^1(M),L^2(M))$, and
        \item $\chi_i \circ |g|^{-\frac{1}{2}}$ is a $C^1$ family of bounded operators on $L^2(M)$.
    \end{enumerate}
    For continuous differentiability of the multiplication operator $|g|^{\frac{1}{2}} \circ g^{jk}$ on $H^1(M)$ it is crucial that $g$ has at least one derivative, since the derivative may hit the metric. This is more than guaranteed by working with $g \in \mathcal G(M)$. The formula for $D\Delta(g)[h]$ follows from the chain and product rule as well as the identities
    \begin{align*}
        D|g|[h] &= |g|\tr_gh, & Dg^{jk}[h] = -h^{jk}.&
    \end{align*}
    By the same argument, $\lambda - \Delta(g)$ is a family of bounded operators from $H^2(M)$ to $L^2(M)$ which depends continuously on both $\lambda$ and $g$.
    
    Now fix $g \in \mathcal G(M)$ and $\lambda \in \rho(M)$. From the continuous dependence of the eigenvalues on the metric \cite[Lemma 4.5.13]{LableeSpectralTheory}, it follows that there exists a neighborhood $\mathcal U$ of $g$ (which can even be chosen to be a $C^0$-norm ball), and $\varepsilon > 0$ such that $(z-\Delta(\tilde g))^{-1} \in \mathcal L(L^2(M),H^2(M))$ is defined for all $\tilde g \in \mathcal U$ and $z \in \mathbb B_\varepsilon(\lambda)$. By \cite[Theorem II.3.1.1]{Hamilton1982}, it follows that $(z-\Delta)^{-1}: \mathbb B_\varepsilon(\lambda) \times \mathcal U \to \mathcal L(L^2(M),H^2(M)$ is continuously differentiable.

    Finally, fix a metric $g \in \mathcal G(M)$ and a curve $\gamma$ in the resolvent set of $g$ as in the statement of the lemma. The spectral projection $P_\gamma$ is given by 
    \begin{align*}
        P_\gamma(\tilde g) = \frac{1}{2\pi i} \int_\gamma (z - \Delta(\tilde g))^{-1} dz,
    \end{align*}
    which is defined whenever $\gamma$ avoids the spectrum on $\tilde g$. In particular, there exists a neighborhood $\mathcal U \subseteq \mathcal G(M)$ of $g$ (which again may be taken as a $C^0$-ball) such that $P_\gamma: \mathcal U \to \mathcal L(L^2(M),H^2(M))$ exists. Continuous differentiability follows from differentiating under the integral using the dominated convergence theorem.
\end{proof}

\newpage

\printbibliography
\newpage
\section*{Data availability statement}
The authors confirm that there is no data associated to this manuscript. This work does not use or generate any datasets.
\section*{Conflict of interest statement}
The authors declare no conflict of interest.
\end{document}